\documentclass[11pt,leqno]{amsart}

\usepackage[a4paper, right=20mm, left=20mm, top=20mm]{geometry}

\usepackage{amsthm}
\usepackage[english]{babel}
\usepackage[utf8]{inputenc}
\usepackage{hyperref}

\usepackage{amsmath}
\usepackage{amssymb}
\usepackage{mathtools}
\usepackage{mathrsfs}
\usepackage{enumerate}

\usepackage{tikz}
\usetikzlibrary{matrix,positioning}

\numberwithin{equation}{section}
\numberwithin{figure}{section}

\theoremstyle{plain}
\newtheorem{Theorem}{Theorem}[section]
\newtheorem*{Theorem*}{Theorem}
\newtheorem{Lemma}[Theorem]{Lemma}
\newtheorem{Proposition}[Theorem]{Proposition}
\newtheorem{Corollary}[Theorem]{Corollary}
\newtheorem{TheoremAlph}{Theorem}

\theoremstyle{definition}
\newtheorem{Definition}[Theorem]{Definition}
\newtheorem*{Definition*}{Definition}
\newtheorem{Remark}[Theorem]{Remark}
\newtheorem*{Remark*}{Remark}

\newcommand{\C}{\mathbb{C}}
\newcommand{\R}{\mathbb{R}}

\newcommand{\Z}{\mathbb{Z}}
\newcommand{\N}{\mathbb{N}}
\newcommand{\kk}{\Bbbk}

\newcommand{\G}{\mathbf{G}}

\newcommand{\Hom}{\mathrm{Hom}}
\newcommand{\End}{\mathrm{End}}

\newcommand{\Cdot}{\boldsymbol{\cdot}} 
\newcommand{\Stab}{\mathrm{Stab}}
\newcommand{\Rep}{\mathrm{Rep}}
\newcommand{\Tilt}{\mathrm{Tilt}}
\newcommand{\GL}{\mathrm{GL}}
\newcommand{\SL}{\mathrm{SL}}
\newcommand{\gfd}{\mathrm{gfd}}
\newcommand{\wfd}{\mathrm{wfd}}
\newcommand{\pr}{\mathrm{pr}}

\providecommand{\abs}[1]{\lvert#1\rvert}

\title[Generic direct summands of tensor products]{Generic direct summands of tensor products\\for simple algebraic groups and quantum groups}
\author{Jonathan Gruber}
\address{Department of Mathematics, University of York, Heslington, York, UK}
\email{jonathan.gruber@york.ac.uk}
\date{\today}

\begin{document}

\begin{abstract}
	Let $\G$ be either a simple linear algebraic group over an algebraically closed field of positive characteristic or a quantum group at a root of unity.
	We define new classes of indecomposable $\G$-modules, which we call \emph{generic direct summands} of tensor products because they appear generically in Krull-Schmidt decompositions of tensor products of simple $\G$-modules and of Weyl modules.
	We establish a Steinberg-Lusztig tensor product theorem for generic direct summands of tensor products of simple $\G$-modules and provide examples of generic direct summands for $\G$ of type $\mathrm{A}_1$ and $\mathrm{A}_2$.
\end{abstract}

\maketitle

\section*{Introduction}

Tensor products are ubiquitous in representation theory, and the problem of finding direct sum decompositions of tensor products of representations has been studied by many mathematicians in many different settings.
Two well-known examples are the \emph{Clebsch-Gordan formula} and the \emph{Littlewood-Richardson rule}, which describe the decompositions of tensor products of irreducible representations of the algebraic groups $\SL_2(\C)$ and $\GL_n(\C)$, respectively, as direct sums of irreducible representations.
The situation becomes more complicated when one replaces the field $\C$ of complex numbers by a field of positive characteristic.
Then, the corresponding categories of representations are no longer semisimple, i.e.\ not every representation can be written as a direct sum of irreducible representations.
Nevertheless, a finite-dimensional representation usually still admits an essentially unique decomposition as a direct sum of indecomposable representations, called a \emph{Krull-Schmidt decomposition}.
In this article, we exhibit certain indecomposable representations of algebraic groups (and quantum groups) which appear \emph{generically} in Krull-Schmidt decompositions of tensor products of certain important representations, and which we therefore call \emph{generic direct summands of tensor products}.

More specifically, let $\G$ be a simply connected simple algebraic group over an algebraically closed field of characteristic $\ell>0$.
(All of our results have analogues for quantum groups at roots of unity, but we focus on algebraic groups in the introduction.)
Then the category $\Rep(\G)$ of finite-dimensional rational representations of $\G$ has simple objects $\{ L(\lambda) \mid \lambda \in X^+\}$ parametrized by the set $X^+$ of dominant weights for $\G$, and every simple $\G$-module $L(\lambda)$ arises as the unique simple quotient of the Weyl module $\Delta(\lambda)$ of highest weight $\lambda$.
We are interested in indecomposable direct summands of tensor products of the form $L(\lambda) \otimes L(\mu)$ or $\Delta(\lambda) \otimes \Delta(\mu)$, for $\lambda,\mu \in X^+$.

Before we state our results, we need to recall some facts about blocks of algebraic groups.
Let $\Phi$ be the root system of $\G$, with (finite) Weyl group $W_\mathrm{fin}$.
Then the affine Weyl group $W_\mathrm{aff} =\Z\Phi \rtimes W_\mathrm{fin}$ acts on the weight lattice $X$ of $\G$ via the so-called $\ell$-dilated dot action.
The \emph{linkage principle} asserts that two simple $\G$-modules $L(\lambda)$ and $L(\mu)$, with $\lambda,\mu \in X^+$, belong to the same block of $\Rep(\G)$ only if $\lambda$ and $\mu$ belong to the same $W_\mathrm{aff}$-orbit.
Let us write $C_\mathrm{fund}$ for the fundamental alcove in $X_\R = X \otimes_\Z \R$ (see Subsection \ref{subsec:linkagetranslation}), so that $\overline{C}_\mathrm{fund} \cap X$ is a system of representatives for the $W_\mathrm{aff}$-orbits in $X$.
For a weight $\lambda \in \overline{C}_\mathrm{fund} \cap X$, the \emph{linkage class} of $\lambda$ is the Serre subcategory of $\Rep(\G)$ generated by the simple $\G$-modules $L(\mu)$ for $\mu \in W_\mathrm{aff} \Cdot \lambda \cap X^+$.
By the linkage principle, the category $\Rep(\G)$ has a decomposition
\[ \Rep(\G) = \bigoplus_{\lambda \in \overline{C}_\mathrm{fund} \cap X} \Rep_\lambda(\G) . \]
Furthermore, there are \emph{translation functors}
\[ T_\lambda^\mu \colon \Rep_\lambda(\G) \longrightarrow \Rep_\mu(\G) \]
for $\lambda,\mu \in \overline{C}_\mathrm{fund} \cap X$, and if $\lambda,\mu \in C_\mathrm{fund} \cap X$ then $T_\lambda^\mu$ is an equivalence with quasi-inverse $T_\mu^\lambda$.
We assume from now on that $\ell \geq h$, the Coxeter number of $\G$, and we call the linkage class $\Rep_0(\G)$ the \emph{principal block} of $\G$.

The classical linkage principle and the translation functors are a priori not compatible with the monoidal structure of $\Rep(\G)$, but a workaround for this problem was recently provided in \cite{GruberLinkageTranslation}.
Namely, there is a tensor ideal of \emph{singular $\G$-modules} in $\Rep(\G)$ such that the essential image of the principal block in the associated quotient category is closed under tensor products.
Furthermore, the monoidal structure of the quotient category is controlled by the resulting monoidal structure on the quotient of the principal block, via a \emph{translation principle for tensor products}.
The non-singular $\G$-modules are called \emph{regular}, and every $\G$-module $M$ has a direct sum decomposition
\[ M \cong M_\mathrm{sing} \oplus M_\mathrm{reg} \]
into its singular part $M_\mathrm{sing}$ and its regular part $M_\mathrm{reg}$.
In order to understand the regular parts of tensor products of simple $\G$-modules or Weyl modules with arbitrary highest weights, it suffices by the results of \cite[Section 3]{GruberLinkageTranslation} to describe the regular parts of tensor products of simple $\G$-modules or Weyl modules in the principal block.
Our first main result is that the regular part of a tensor product of Weyl modules in the principal block consists of a single indecomposable $\G$-module.

\begin{TheoremAlph} \label{thm:introWeyl}
	Let $x,y \in W_\mathrm{aff}$ such that $x\Cdot0,y\Cdot0 \in X^+$.
	Then the tensor product $\Delta(x\Cdot0) \otimes \Delta(y\Cdot0)$ has a unique regular indecomposable direct summand $G_\Delta(x,y)$.
\end{TheoremAlph}

For $\lambda,\mu \in C_\mathrm{fund} \cap X$, the translation principle for tensor products from \cite[Theorem 3.14]{GruberLinkageTranslation} implies that
\[ \big( \Delta(x\Cdot\lambda) \otimes \Delta(y\Cdot\mu) \big)_\mathrm{reg} \cong \bigoplus_{\nu \in C_\mathrm{fund} \cap X} T_0^\nu G_\Delta(x,y)^{\oplus c_{\lambda,\mu}^\nu} , \]
where $c_{\lambda,\mu}^\nu$ are the structure constants of the Verlinde algebra of $\G$ (see Theorem \ref{thm:translationtensor} below).
In other words, up to applying translation functors, the $\G$-module $G_\Delta(x,y)$ appears \emph{generically} in Krull-Schmidt decompositions of tensor products of the form $\Delta(x\Cdot\lambda) \otimes \Delta(y\Cdot\mu)$.
We call $G_\Delta(x,y)$ the \emph{generic direct summand} of $\Delta(x\Cdot0) \otimes \Delta(y\Cdot0)$.

The tensor product of two simple $\G$-modules in $\Rep_0(\G)$ may in general have more than one regular indecomposable direct summand, but we still get a unique regular indecomposable direct summand of maximal \emph{good filtration dimension} (see Equation \eqref{eq:gfdwfd} below).
Writing $\ell \colon W_\mathrm{aff} \to \Z_{\geq 0}$ for the length function on $W_\mathrm{aff}$, our second main result is as follows:

\begin{TheoremAlph} \label{thm:introsimple}
	Let $x,y \in W_\mathrm{aff}$ such that $x\Cdot0,y\Cdot0 \in X^+$.
	Then the tensor product $L(x\Cdot0) \otimes L(y\Cdot0)$ has a unique regular indecomposable direct summand $G(x,y)$ of good filtration dimension $\ell(x)+\ell(y)$.
\end{TheoremAlph}

As before, the $\G$-module $T_0^\nu G(x,y)$ appears generically in Krull-Schmidt decompositions of tensor products of the form $L(x\Cdot\lambda) \otimes L(y\Cdot\mu)$, for $\lambda,\mu,\nu \in C_\mathrm{fund} \cap X$.
We call $G(x,y)$ the \emph{generic direct summand} of $L(x\Cdot0) \otimes L(y\Cdot0)$.

We believe that describing the structure of generic direct summands will be a key step towards understanding the monoidal structure of $\Rep(\G)$.
An important tool for finding such descriptions is an analogue of \emph{Steinberg's tensor product theorem} for generic direct summands of tensor products of simple $\G$-modules, which we explain below.
Let us write $X_1$ for the set of $\ell$-restricted weights in $X^+$ (see Section \ref{sec:SteinbergLusztigTPtheorem}), and note that every weight $\lambda \in X^+$ has a decomposition $\lambda = \lambda_0 + \ell \lambda_1$ with $\lambda_0 \in X_1$ and $\lambda_1 \in X^+$.
By Steinberg's tensor product theorem, the simple $\G$-module $L(\lambda)$ admits a tensor product decomposition
\[ L(\lambda) \cong L(\lambda_0) \otimes L(\lambda_1)^{[1]} , \]
where $M \mapsto M^{[1]}$ denotes the Frobenius twist functor on $\Rep(\G)$.

Now suppose that $\lambda = x \Cdot 0$ for some $x \in W_\mathrm{aff}$, and let $\mu = \mu_0 + \ell\mu_1 = y\Cdot0$ with $\mu_0 \in X_1$, $\mu_1 \in X^+$ and $y \in W_\mathrm{aff}$.
Then we can write $\lambda_0 = x_0 \Cdot 0$ and $\mu_0 = y_0 \Cdot 0$ for certain elements $x_0 , y_0 \in W_\mathrm{ext}$ of the \emph{extended affine Weyl group} $W_\mathrm{ext} = X \rtimes W_\mathrm{fin}$ of $\G$.
The theory of generic direct summands described above readily extends to tensor products of simple $\G$-modules with highest weights in the $W_\mathrm{ext}$-orbit of $0$ (rather than just the $W_\mathrm{aff}$-orbit), and this allows us to state the following analogue of the Steinberg tensor product theorem:

\begin{TheoremAlph} \label{thm:introSteinberg}
	Let $x,y \in W_\mathrm{ext}$ with $x\Cdot0 \in X^+$ and $y\Cdot0 \in X^+$, and let $x_0,y_0 \in W_\mathrm{ext}$ and $\lambda,\mu \in X^+$ such that
	\[ x\Cdot0 = x_0\Cdot0 + \ell\lambda , \quad x_0\Cdot0 \in X_1 , \qquad y\Cdot0 = y_0\Cdot0 + \ell\mu , \quad y_0\Cdot0 \in X_1 . \]
	Further denote by $M(\lambda,\mu)$ the unique indecomposable direct summand of $L(\lambda) \otimes L(\mu)$ with a non-zero $(\lambda+\mu)$-weight space.
	Then $G(x,y)$ is a direct summand of $G(x_0,y_0) \otimes M(\lambda,\mu)^{[1]}$.
	If $G(x_0,y_0)$ is indecomposable as a module over the first Frobenius kernel of $\G$ then
	\[ G(x,y) \cong G(x_0,y_0) \otimes M(\lambda,\mu)^{[1]} . \]
\end{TheoremAlph}

The theorem is used in Section \ref{sec:smallrank} to give a complete description of the generic direct summands of tensor products of simple $\G$-modules for $\G$ of type $\mathrm{A}_1$ and $\mathrm{A}_2$.
It is important to point out that the condition for the isomorphism $G(x,y) \cong G(x_0,y_0) \otimes M(\lambda,\mu)^{[1]}$ is slightly different in the quantum analogue of Theorem \ref{thm:introSteinberg} (see Theorem \ref{thm:lusztigtensorproductgenericdirectsummand}).

To conclude the introduction, we give a an outline of the contents of this article.
Section \ref{sec:preliminaries} is a summary of some important results about the representation theory of $\G$, including the classical linkage principle and translation principle.
In Section \ref{sec:minimalcomplexessingularmodules} we recall some of the main results of \cite{GruberLinkageTranslation} about minimal tilting complexes, singular $\G$-modules and the linkage principle and translation principle for tensor products.
Generic direct summands of tensor products are introduced and studied in Section~\ref{sec:genericdirectsummands}; see Theorems \ref{thm:genericdirectsummandWeylmodule} and \ref{thm:genericdirectsummandsimplemodule} for the proofs of Theorems \ref{thm:introWeyl} and \ref{thm:introsimple}, respectively.
In Section \ref{sec:SteinbergLusztigTPtheorem}, we study the relation between Steinberg's tensor product theorem and generic direct summands of tensor products of simple $\G$-modules, and we prove Theorem \ref{thm:introSteinberg} in Theorem \ref{thm:steinbergtensorproductgenericdirectsummand} (and its quantum analogue in Theorem \ref{thm:lusztigtensorproductgenericdirectsummand}).
Finally, in Section \ref{sec:smallrank}, we provide examples of generic direct summands of tensor products of simple $\G$-modules and of dual Weyl modules for $\G$ of type $\mathrm{A}_1$ and $\mathrm{A}_2$.

\subsection*{Acknowledgements}

This article contains some of the main results of my PhD thesis, and I would like to express my deepest gratitude to my PhD advisor Donna Testerman.
I would also like to thank Stephen Donkin, Daniel Nakano and Thorge Jensen for helpful discussions and comments.
This work was supported by the Singapore MOE grant R-146-000-294-133 and by the Swiss National Science Foundation under the grants FNS 200020\_175571, FNS 200020\_207730 and P500PT\_206751.

\section{Preliminaries} \label{sec:preliminaries}

We follow the notation and conventions of \cite{GruberLinkageTranslation}, which we recall in this section.

\subsection{Roots and weights}

Let $\Phi$ be a simple root system in a euclidean space $X_\R$ with scalar product $( - \,, - )$. For $\alpha\in\Phi$, we denote by $\alpha^\vee = 2 \alpha / ( \alpha , \alpha )$ the \emph{coroot} of $\alpha$ and by $s_\alpha \in \GL(X_\R)$ the reflection corresponding to $\alpha$, so $s_\alpha(x) = x - (x , \alpha^\vee) \cdot \alpha$ for $x\in X_\R$.
The \emph{weight lattice} of $\Phi$ is
\[ X \coloneqq \{ \lambda \in X_\R \mid ( \lambda , \alpha^\vee ) \in \Z \text{ for all } \alpha \in \Phi \} , \]
and the \emph{Weyl group} of $\Phi$ is the (finite) subgroup $W_\mathrm{fin} = \langle s_\alpha \mid \alpha\in\Phi \rangle$ of $\GL(X_\R)$ generated by the reflections $s_\alpha$.
The index of the root lattice $\Z\Phi$ in the weight lattice $X$ is finite, and we call $X / \Z\Phi$ the \emph{fundamental group} of $\Phi$.
Now fix a positive system $\Phi^+ \subseteq \Phi$ corresponding to a base $\Pi$ of $\Phi$, and let
\[ X^+ \coloneqq \{ \lambda \in X \mid (\lambda , \alpha^\vee) \geq 0 \text{ for all } \alpha \in \Phi^+ \} \]
be the set of \emph{dominant weights} with respect to $\Phi^+$.
We consider the partial order on $X$ that is defined by $\lambda \geq  \mu$ if and only if $\lambda-\mu$ is a non-negative integer linear combination of positive roots.
Furthermore, we write $\tilde{\alpha}_\mathrm{h}$ for the highest root and $\alpha_\mathrm{h}$ for the highest short root in $\Phi^+$, with the convention that $\tilde{\alpha}_\mathrm{h}=\alpha_\mathrm{h}$ (and that all roots are short) if $\Phi$ is simply laced.
The \emph{height} of an element $\gamma= \sum_{\alpha\in\Pi} c_\alpha \cdot \alpha$ of the root lattice is $\mathrm{ht}(\gamma)= \sum_{\alpha\in\Pi} c_\alpha$.
We let $\rho = \frac12 \sum_{\alpha\in\Phi^+} \alpha$ be the half sum of all positive roots and write $h = ( \rho , \alpha_\mathrm{h}^\vee ) + 1$ for the \emph{Coxeter number} of $\Phi$.
The \emph{dot action} of $W_\mathrm{fin}$ on $X_\R$ is defined by
\[ w \Cdot x = w(x+\rho)-\rho \]
for $w\in W_\mathrm{fin}$ and $x\in X_\R$.
The set of \emph{simple reflections} $S_\mathrm{fin} = \{ s_\alpha \mid \alpha\in\Pi \}$ with respect to $\Pi$ is a minimal generating set of $W_\mathrm{fin}$, and $(W_\mathrm{fin},S_\mathrm{fin})$ is a Coxeter system.
As $W_\mathrm{fin}$ is finite, there exists a unique longest element $w_0 \in W_\mathrm{fin}$ with respect to $S_\mathrm{fin}$.
For every simple root $\alpha\in\Pi$, we write $\varpi_\alpha\in X$ for the \emph{fundamental dominant weight} with $(\varpi_\alpha , \alpha^\vee)=1$ and $(\varpi_\alpha , \beta^\vee)=0$ for $\alpha \neq \beta \in \Pi$.

\subsection{Algebraic groups and quantum groups} \label{sec:algebraicgroupsquantumgroups}

The categories of representations of simple algebraic groups over fields of positive characteristic and of quantum groups at roots of unity have many structural properties in common.
We will treat the two cases simultaneously, referring to the former as the \emph{modular case} and to the latter as the \emph{quantum case} when a distinction is necessary.

In the modular case, $\G$ denotes a simply-connected simple algebraic group scheme over an algebraically closed field $\kk$ of characteristic $\ell>0$.
We assume that $\G$ arises by extension of scalars from a split simply-connected simple algebraic group scheme $\G_\Z$ over $\Z$, with a split maximal torus such that the corresponding root system is isomorphic to $\Phi$.
We write $\Rep(\G)$ for the category of finite-dimensional $\G$-modules, in the sense of Section I.2.3 in \cite{Jantzen}.

In the quantum case, $\G = U_\zeta(\mathfrak{g})$ denotes the specialization at a complex $\ell$-th root of unity $\zeta \in \C$ of Lusztig's integral form (with divided powers) of the quantum group corresponding to the complex simple Lie algebra $\mathfrak{g}$ with root system $\Phi$, as in \cite[Section 1]{GruberLinkageTranslation} or \cite[Chapter II.H]{Jantzen}.
We write $\Rep(\G)$ for the category of finite-dimensional $\G$-modules of type $1$; see \cite[Section II.H.10]{Jantzen}.

In either case, $\Rep(\G)$ is a highest weight category with weight poset $(X^+,\leq)$, and we write
\[ L(\lambda) , \qquad \Delta(\lambda) , \qquad \nabla(\lambda) , \qquad T(\lambda) \]
for the simple $\G$-module, the Weyl module, the dual Weyl module and the indecomposable tilting module of highest weight $\lambda \in X^+$.
Furthermore, $\Rep(\G)$ is a rigid braided monoidal category, and the the tensor product of any two tilting $\G$-modules is a tilting $\G$-module \cite{DonkinGoodFiltration,MathieuGoodFiltration,ParadowskiGoodFiltration}.
The full subcategory of tilting $\G$-modules in $\Rep(\G)$ is denoted by $\Tilt(\G)$.
By \cite[Section II.2.12]{Jantzen} and \cite[Section 2.A]{AndersenStroppelTubbenhauer}, there is a character preserving duality functor $M \mapsto M^\tau$ on $\Rep(\G)$, and for $\lambda \in X^+$, we have $L(\lambda)^\tau \cong L(\lambda)$, $\Delta(\lambda)^\tau \cong \nabla(\lambda)$ and $T(\lambda)^\tau \cong T(\lambda)$.

\subsection{Linkage and translation} \label{subsec:linkagetranslation}

Let us write $W_\mathrm{aff} = \Z\Phi \rtimes W_\mathrm{fin}$ and $W_\mathrm{ext} = X \rtimes W_\mathrm{fin}$ for the affine Weyl group and the extended affine Weyl group of $\G$, respectively.
For $\gamma \in X$ and $w \in W_\mathrm{fin}$, the corresponding element of $W_\mathrm{ext}$ is denoted by $t_\gamma w$.
The $\ell$-dilated dot action of $W_\mathrm{ext}$ on $X_\R$ is defined by
\[ t_\gamma w \Cdot x = w(x+\rho)-\rho + \ell\gamma , \]
for $x \in X_\R$.
For $\beta \in \Phi^+$ and $r \in \Z$, the fixed points of the affine reflection $s_{\beta,r} = t_{r\beta} s_\beta$ with respect to the $\ell$-dilated dot action form an affine hyperplane $H_{\beta,r} = \{ x \in X_\R \mid (x+\rho,\beta^\vee) = \ell r \}$, and any connected component of $X_\R \setminus \big( \bigcup\nolimits_{\beta,r} H_{\beta,r} \big)$ is called an alcove.
The fundamental alcove is
\[ C_\mathrm{fund} \coloneqq \big\{ x \in X_\R \mathrel{\big|} 0 < (x+\rho,\beta^\vee) < \ell \text{ for all } \beta \in \Phi^+ \big\} . \]
Its closure $\overline{C}_\mathrm{fund}$ is a fundamental domain for the $\ell$-dilated dot action of $W_\mathrm{aff}$ on $X_\R$, and $\overline{C}_\mathrm{fund} \cap X$ is a fundamental domain for the action of $W_\mathrm{aff}$ on $X$.
A weight $\lambda \in X$ is called $\ell$-regular if it belongs to an alcove, and $\ell$-singular if it belongs to a reflection hyperplane.
Note that $\ell$-regular weights exist if and only if $\ell \geq (\rho,\alpha_\mathrm{h}^\vee) + 1 = h$ (the Coxeter number of $\G$), if and only if $0 \in C_\mathrm{fund}$.

The affine Weyl group $W_\mathrm{aff}$ is a Coxeter group with simple reflections $S = \{ s_\alpha \mid \alpha \in \Pi \} \sqcup \{ s_{\alpha_\mathrm{h},1} \}$.
For $w \in W_\mathrm{ext}$, we write $\ell(w)$ for the number of reflection hyperplanes separating $C_\mathrm{fund}$ and $w \Cdot C_\mathrm{fund}$.
Then the restriction to $W_\mathrm{aff}$ of $\ell \colon W_\mathrm{ext} \to \Z_{\geq 0}$ is the usual length function on $W_\mathrm{aff}$ with respect to $S$, and we have
\[ \Omega \coloneqq \{ w \in W_\mathrm{ext} \mid \ell(w) = 0 \} = \Stab_{W_\mathrm{ext}}(C_\mathrm{fund}) . \]
Since $W_\mathrm{aff}$ acts simply transitively on the set of alcoves, we also have a semidirect product decomposition $W_\mathrm{ext} = W_\mathrm{aff} \rtimes \Omega$, and $\Omega \cong W_\mathrm{ext} / W_\mathrm{aff} \cong X / \Z\Phi$.
For every element $w \in W_\mathrm{ext}$, the coset $W_\mathrm{fin} w$ has a unique element of minimal length, and we write
\[ W_\mathrm{ext}^+ \coloneqq \{ w \in W_\mathrm{ext} \mid w \text{ has minimal length in } W_\mathrm{fin} w \} \] and $W_\mathrm{aff}^+ = W_\mathrm{ext}^+ \cap W_\mathrm{aff}$.
If $\ell \geq h$ and $w \in W_\mathrm{ext}$ then $w \in W_\mathrm{ext}^+$ if and only if $w \Cdot 0 \in X^+$.

Suppose from now on that $\ell \geq h$.
For $\lambda \in \overline{C}_\mathrm{fund} \cap X$, the linkage class $\Rep_\lambda(\G)$ is the Serre subcategory of $\Rep(\G)$ generated by the simple $\G$-modules $L(\mu)$ with $\mu \in W_\mathrm{aff}\Cdot\lambda \cap X^+$.
The linkage principle (see Chapter II.6 in \cite{Jantzen}) asserts that there is a canonical decomposition
\[ \Rep(\G) = \bigoplus_{\lambda \in \overline{C}_\mathrm{fund} \cap X} \Rep_\lambda(\G) \]
with projection functors
\[ \pr_\lambda \colon \Rep(\G) \longrightarrow \Rep_\lambda(\G) . \]
The linkage class $\Rep_0(\G)$ is called the principal block of $\G$, and the extended principal block is
\[ \Rep_{\Omega\Cdot0}(\G) = \bigoplus_{\lambda \in \Omega\Cdot0} \Rep_\lambda(\G) . \]
Note that $\{ L(w \Cdot 0) \mid w \in W_\mathrm{aff}^+ \}$ is a set of representatives for the simple objects in $\Rep_0(\G)$. 
By the translation principle from Chapter II.7 in \cite{Jantzen}, there are translation functors
\[ T_\lambda^\mu = \pr_\mu\big( L(\nu) \otimes - \big)  \colon \Rep_\lambda(\G) \longrightarrow \Rep_\mu(\G) \]
for $\lambda,\mu \in \overline{C}_\mathrm{fund} \cap X$, where $\nu$ is the unique dominant weight in the $W_\mathrm{fin}$-orbit of $\mu-\lambda$.
If $\lambda$ and $\mu$ are $\ell$-regular then $T_\lambda^\mu$ is an equivalence with quasi-inverse $T_\mu^\lambda$, and for $w \in W_\mathrm{aff}^+$, we have
\[ T_\lambda^\mu L(w\Cdot\lambda) \cong L(w\Cdot\mu) , \quad T_\lambda^\mu \nabla(w\Cdot\lambda) \cong \nabla(w\Cdot\mu) , \quad T_\lambda^\mu \Delta(w\Cdot\lambda) \cong \Delta(w\Cdot\mu) , \quad T_\lambda^\mu T(w\Cdot\lambda) \cong T(w\Cdot\mu) . \]
For $\omega \in \Omega$, we define an auto-equivalence $T^\omega$ of $\Rep_{\Omega\Cdot0}(\G)$ by $T^\omega = \bigoplus\nolimits_{\lambda \in \Omega\Cdot0} T_{\lambda}^{\omega\Cdot\lambda}$.

\section{Minimal tilting complexes and singular \texorpdfstring{$\G$}{G}-modules} \label{sec:minimalcomplexessingularmodules}

The canonical functor $\mathfrak{T} \colon K^b\big( \Tilt(\G) \big) \longrightarrow D^b\big( \Rep(\G) \big)$ from the homotopy category of $\Tilt(\G)$ to the derived category of $\Rep(\G)$ is an equivalence of triangulated categories.
As explained in \cite{GruberMinimalTilting}, this allows us to define for every $\G$-module $M$ a \emph{minimal tilting complex} $C_\mathrm{min}(M)$.
Up to isomorphism of complexes, it is the unique minimal complex $C$ in $\Tilt(\G)$ with cohomology groups $H^0(C) \cong M$ and $H^i(C) = 0$ for $i \neq 0$.
The minimal tilting complex $C_\mathrm{min}(M)$ encodes a lot of important information about the $\G$-module $M$.
For instance, if $M$ is non-zero then two invariants called the \emph{good filtration dimension} and the \emph{Weyl filtration dimension} of $M$ can be determined as
\begin{equation} \label{eq:gfdwfd}
	\gfd(M) = \max\big\{ i \in \Z \mathrel{\big|} C_\mathrm{min}(M)_i \neq 0 \big\} , \qquad \qquad \wfd(M) = - \min\big\{ i \in \Z \mathrel{\big|} C_\mathrm{min}(M)_i \neq 0 \big\} ,
\end{equation}
respectively; see Lemma 2.15 in \cite{GruberMinimalTilting}.
The good filtration dimension of $L(x\Cdot\lambda)$ and of $\Delta(x\Cdot\lambda)$ is given by the length $\ell(x)$, for $x \in W_\mathrm{ext}^+$ and $\lambda \in C_\mathrm{fund} \cap X$, by \cite[Remark 4.2]{GruberLinkageTranslation} or \cite{ParkerGFD}.

A partial description of the minimal tilting complexes of Weyl modules and simple $\G$-modules of $\ell$-regular highest weight is given in Propositions 2.4 and 2.5 in \cite{GruberLinkageTranslation}; we (partially) recall these results here.
Recall that $W_\mathrm{ext} = W_\mathrm{aff} \rtimes \Omega$, and write $x \mapsto \omega_x$ for the canonical homomorphism $W_\mathrm{ext} \to \Omega$ with kernel $W_\mathrm{aff}$.
A tilting module $T \cong \bigoplus\nolimits_{\lambda \in X^+} T(\lambda)^{\oplus m_\lambda}$ is \emph{negligible} if $m_\lambda = 0$ for all $\lambda \in C_\mathrm{fund} \cap X$.

\begin{Proposition} \label{prop:minimalcomplexWeylmodule}
	Let $x\in W_\mathrm{ext}^+$ and $\lambda\in C_\mathrm{fund} \cap X$, and write $C_\mathrm{min}\big( \Delta(x\Cdot \lambda) \big)$ as
	\[ \cdots \xrightarrow{~d_{-2}~} T_{-1} \xrightarrow{~d_{-1}~} T_0 \xrightarrow{~\,d_0\,~} T_1 \xrightarrow{~\,d_1\,~} \cdots . \]
	Then
	\begin{enumerate}
		\item $T_i=0$ for all $i<0$ and all $i>\ell(x)$;
		\item $T_0 \cong T(x\Cdot \lambda)$ and $T_{\ell(x)} \cong T(\omega_x\Cdot\lambda)$;
		\item $T_i$ is negligible for all $i \neq \ell(x)$.
	\end{enumerate}
\end{Proposition}

\begin{Proposition} \label{prop:minimalcomplexsimplemodule}
	Let $x\in W_\mathrm{ext}^+$ and $\lambda\in C_\mathrm{fund} \cap X$, and write $C_\mathrm{min}\big( L(x\Cdot \lambda) \big)$ as
	\[ \cdots \xrightarrow{~d_{-2}~} T_{-1} \xrightarrow{~d_{-1}~} T_0 \xrightarrow{~\,d_0\,~} T_1 \xrightarrow{~\,d_1\,~} \cdots . \]
	Then
	\begin{enumerate}
		\item $T_i \cong T_{-i}$ for all $i\in\Z$;
		\item $T_i=0$ for all $i\in\Z$ with $\abs{i}>\ell(x)$;
		\item $[T_0 : T(x\Cdot \lambda)]_\oplus = 1$ and $T_{\ell(x)} \cong T_{-\ell(x)} \cong T(\omega_x\Cdot\lambda)$;
		\item $T_{\ell(x)-1} \cong T_{1-\ell(x)}$ is negligible.
	\end{enumerate}
\end{Proposition}

The set of negligible tilting modules is stable under direct sums and retracts, and under tensor products with arbitrary tilting $\G$-modules (see \cite{GeorgievMathieuFusion,AndersenParadowskiFusionCategories,EtingofOstrikSemisimplification}), so it forms a \emph{thick tensor ideal} in $\Tilt(\G)$.
As in \cite{GruberLinkageTranslation}, we call a $\G$-module $M$ \emph{singular} if all terms of $C_\mathrm{min}(M)$ are negligible,
and \emph{regular} otherwise.
The singular $\G$-modules form a thick tensor ideal in $\Rep(\G)$ by Lemma 2.3 in \cite{GruberTensorIdeals}, that is, for $\G$-modules $M$ and $N$, we have
\begin{enumerate}
	\item $M \oplus N$ is singular if and only if $M$ and $N$ are both singular;
	\item if $M$ is singular then so is $M \otimes N$.
\end{enumerate}
Furthermore, a simple $\G$-module $L(\lambda)$ or Weyl module $\Delta(\lambda)$ of highest weight $\lambda \in X^+$ is regular if and only if $\lambda$ is $\ell$-regular; see Lemma 3.3 in \cite{GruberLinkageTranslation}.

Following Definition 3.10 in \cite{GruberLinkageTranslation}, every $\G$-module $M$ has a direct sum decomposition
\[ M \cong M_\mathrm{sing} \oplus M_\mathrm{reg} , \]
where $M_\mathrm{sing}$ is the direct sum of all singular indecomposable direct summands of $M$ and $M_\mathrm{reg}$ is the direct sum of all regular indecomposable direct summands of $M$, for a fixed Krull-Schmidt decomposition of $M$.
Note that $M_\mathrm{sing}$ and $M_\mathrm{reg}$ are determined uniquely up to isomorphism, but the direct sum decomposition is not canonical of functorial.
The following `linkage principle for tensor products' is Lemma 3.12 in \cite{GruberLinkageTranslation}.

\begin{Lemma} \label{lem:linkagetensor}
	Let $\lambda \in C_\mathrm{fund}$ and $\omega \in \Omega$, and let $M$ and $N$ be $\G$-modules in the linkage classes $\Rep_\lambda(\G)$ and $\Rep_{\omega\Cdot0}(\G)$, respectively.
	Then $(M \otimes N)_\mathrm{reg}$ belongs to $\Rep_{\omega\Cdot\lambda}(\G)$.
\end{Lemma}

As an immediate consequence of Lemma \ref{lem:linkagetensor}, we see that if $M$ and $N$ are $\G$-modules in $\Rep_0(\G)$ then so is $(M \otimes N)_\mathrm{reg}$, and similarly for $\Rep_{\Omega\Cdot0}(\G)$.
We also have a `translation principle for tensor products', see Theorem 3.14 in \cite{GruberLinkageTranslation}.

\begin{Theorem} \label{thm:translationtensor}
	Let $\lambda,\mu \in C_\mathrm{fund} \cap X$ and let $M$ and $N$ be $\G$-modules in $\Rep_0(\G)$.
	Then
	\[ ( T_0^\lambda M \otimes T_0^\mu N )_\mathrm{reg} \cong \bigoplus_{\nu \in C_\mathrm{fund} \cap X} T_0^\nu (M \otimes N)_\mathrm{reg}^{\oplus c_{\lambda,\mu}^\nu} , \]
	where $c_{\lambda,\mu}^\nu$ denotes the multiplicity of $T(\nu)$ in a Krull-Schmidt decomposition of $T(\lambda) \otimes T(\mu)$.%
	\footnote{The multiplicities $c_{\lambda,\mu}^\nu$ are the structure constants of the \emph{Verlinde algebra} of $\G$, i.e.\ the split Grothendieck group of the quotient category of $\Tilt(\G)$ by the tensor ideal of negligible tilting modules \cite{GeorgievMathieuFusion,AndersenParadowskiFusionCategories}.}
\end{Theorem}

Recall the autoequivalences $T^\omega$ of $\Rep_{\Omega\Cdot0}(\G)$ introduced at the end of Section \ref{sec:preliminaries}.
As a special case of Theorem \ref{thm:translationtensor}, we get the following result (see Lemma 3.15 in \cite{GruberLinkageTranslation}):

\begin{Lemma} \label{lem:translationtensorfundamentalgroup}
	Let $M$ and $N$ be $\G$-modules in $\Rep_{\Omega\Cdot0}(\G)$ and let $\omega,\omega^\prime \in \Omega$.
	Then
	\[ \big( T^\omega M \otimes T^{\omega^\prime} N \big)_\mathrm{reg} \cong T^{\omega\omega^\prime} ( M \otimes N )_{\mathrm{reg}} . \]
\end{Lemma}

\section{Generic direct summands} \label{sec:genericdirectsummands}

In this section, we study the regular parts of tensor products of Weyl modules and simple $\G$-modules.
Recall that we write $x \mapsto \omega_x$ for the canonical homomorphism $W_\mathrm{ext} \to \Omega$ with kernel $W_\mathrm{aff}$.

\begin{Theorem} \label{thm:genericdirectsummandWeylmodule}
	Let $x,y\in W_\mathrm{ext}^+$. Then the tensor product $\Delta(x\Cdot0) \otimes \Delta(y\Cdot0)$ has a unique regular indecomposable direct summand $G_\Delta(x,y)$.
\end{Theorem}
\begin{proof}
	First suppose that $x,y \in W_\mathrm{aff}^+$.
	By Lemma 1.3 in \cite{GruberTensorIdeals} and its proof, the minimail tilting complex $C_\mathrm{min}\big( \Delta(x\Cdot 0) \otimes \Delta(y\Cdot 0) \big)$ is the minimal complex of $C_\mathrm{min}\big( \Delta(x\Cdot 0) \big) \otimes C_\mathrm{min}\big( \Delta(y\Cdot 0) \big)$, and there are split embeddings
	\[ C_\mathrm{min}\big( \Delta(x\Cdot 0) \otimes \Delta(y\Cdot 0) \big)_i \stackrel{\oplus}{\subseteq} \bigoplus_{j+k=i} C_\mathrm{min}\big( \Delta(x\Cdot0) \big)_j \otimes C_\mathrm{min}\big( \Delta(y\Cdot0) \big)_k \]
	for all $i \in \Z$.
	Recall from Proposition \ref{prop:minimalcomplexWeylmodule} that the terms $C_\mathrm{min}\big( \Delta(x\Cdot 0) \big)_j$ of the minimal tilting complex of $\Delta(x\Cdot0)$ are negligible for $j<\ell(x)$ and zero for $j>\ell(x)$, and similarly $C_\mathrm{min}\big( \Delta(y\Cdot0) \big)_k$ is negligible for $k<\ell(y)$ and zero for $k>\ell(y)$.
	Furthermore, we have
	\[ C_\mathrm{min}\big( \Delta(x\Cdot 0) \big)_{\ell(x)} \cong T(0) \qquad \text{and} \qquad C_\mathrm{min}\big( \Delta(y\Cdot0) \big)_{\ell(y)} \cong T(0) . \]
	Combining the above observations, we see that the terms $C_\mathrm{min}\big( \Delta(x\Cdot 0) \otimes \Delta(y\Cdot 0) \big)_i$ of the minimal tilting complex of $\Delta(x\Cdot0) \otimes \Delta(y\Cdot0)$ are negligible for $i<\ell(x)+\ell(y)$ and zero for $i>\ell(x)+\ell(y)$.
	The term in degree $\ell(x)+\ell(y)$ is a direct summand of the tensor product
	\[ C_\mathrm{min}\big( \Delta(x\Cdot 0) \big)_{\ell(x)} \otimes C_\mathrm{min}\big( \Delta(y\Cdot 0) \big)_{\ell(y)} \cong T(0) \otimes T(0) \cong T(0) . \]
	As the terms of the tensor product complex $C_\mathrm{min}\big( \Delta(x\Cdot 0) \big) \otimes C_\mathrm{min}\big( \Delta(y\Cdot 0) \big)$ in degrees $\ell(x)+\ell(y)-1$ and $\ell(x)+\ell(y)+1$ are negligible or zero, respectively, Corollary 2.8 in \cite{GruberMinimalTilting} implies that
	\[ C_\mathrm{min}\big( \Delta(x\Cdot0) \otimes \Delta(y\Cdot0) \big)_{\ell(x)+\ell(y)} \cong T(0) . \]
	Now fix a Krull-Schmidt decomposition
	\[ \Delta(x\Cdot 0) \otimes \Delta(y\Cdot 0) \cong M_1 \oplus \cdots \oplus M_r \]
	and note that
	\[ C_\mathrm{min}\big( \Delta(x\Cdot 0) \otimes \Delta(y\Cdot 0) \big) \cong C_\mathrm{min}(M_1) \oplus \cdots \oplus C_\mathrm{min}(M_r) . \]
	It follows that there exists a unique $k \in \{ 1,\ldots,r \}$ with $C_\mathrm{min}(M_k)_{\ell(x)+\ell(y)} \cong T(0)$, and all of the terms $C_\mathrm{min}(M_i)_j$ of the minimal tilting complexes $C_\mathrm{min}(M_i)$ are negligible for $i \neq k$ (or $i=k$ and $j<\ell(x)+\ell(y)$) and zero for $j>\ell(x)+\ell(y)$ (or $j \geq \ell(x)+\ell(y)$ and $i \neq k$).
	In particular, $M_k$ is the unique regular indecomposable direct summand of $\Delta(x\Cdot0) \otimes \Delta(y\Cdot0)$,
	and $\gfd(M_k) = \ell(x)+\ell(y)$ by Equation \eqref{eq:gfdwfd}.
	
	For arbitrary $x,y \in W_\mathrm{ext}^+$, we can write $x = x_0 \omega_x$ and $y = y_0 \omega_y$ with $x_0,y_0 \in W_\mathrm{aff}^+$, and then
	\[ \Delta(x\Cdot0) \cong T^{\omega_x} \Delta(x_0\Cdot0) \qquad \text{and} \qquad \Delta(y\Cdot0) \cong T^{\omega_y} \Delta(y_0\Cdot0) . \]
	By Lemma \ref{lem:translationtensorfundamentalgroup}, we have
	\begin{multline*}
	\qquad \big( \Delta(x\Cdot0) \otimes \Delta(y\Cdot0) \big)_\mathrm{reg} \cong \big( T^{\omega_x} \Delta(x_0\Cdot0) \otimes T^{\omega_y} \Delta(y_0\Cdot0) \big)_\mathrm{reg} \\ \cong T^{\omega_{xy}} \big( \Delta(x_0\Cdot0) \otimes \Delta(y_0\Cdot0) \big)_\mathrm{reg} \cong T^{\omega_{xy}} G_\Delta(x_0,y_0) , \qquad
	\end{multline*}
	and the claim follows with $G_\Delta(x,y) = T^{\omega_{xy}} G_\Delta(x_0,y_0)$.
\end{proof}

\begin{Remark} \label{rem:genericdirectsummandWeylmodulelastterm}
	Observe that for $x,y \in W_\mathrm{aff}^+$, the proof of Theorem \ref{thm:genericdirectsummandWeylmodule} implies that
	\begin{equation} \label{eq:minimaltiltingcomplexgenericdirectsummandWeyllastterm}
	C_\mathrm{min}\big( G_\Delta(x,y) \big)_{\ell(x)+\ell(y)} \cong T(0) ,
	\end{equation}
	and that $G_\Delta(x,y)$ has good filtration dimension $\ell(x) + \ell(y)$ by \eqref{eq:gfdwfd}.
	In fact, $G_\Delta(x,y)$ is the unique indecomposable direct summand of $\Delta(x\Cdot0) \otimes \Delta(y\Cdot0)$ with good filtration dimension $\ell(x)+\ell(y)$.
	Furthermore, $G_\Delta(x,y)$ belongs to $\Rep_0(\G)$, either by \eqref{eq:minimaltiltingcomplexgenericdirectsummandWeyllastterm} or by Lemma \ref{lem:linkagetensor}.

	For $x,y\in W_\mathrm{ext}^+$, the proof of Theorem \ref{thm:genericdirectsummandWeylmodule} and the above observations imply that
	\[ C_\mathrm{min}\big( G_\Delta(x,y) \big)_{\ell(x)+\ell(y)} \cong T(\omega_{xy}\Cdot0) , \]
	and that $G_\Delta(x,y)$ has good filtration dimension $\ell(x)+\ell(y)$ and belongs to $\Rep_{\omega_{xy}\Cdot0}(\G)$.
\end{Remark}

\begin{Remark} \label{rem:genericdirectsummandinducedmodule}
	For $x,y\in W_\mathrm{ext}^+$, one can show as in Theorem \ref{thm:genericdirectsummandWeylmodule} that $\nabla(x\Cdot0) \otimes \nabla(y\Cdot0)$ has a unique regular indecomposable direct summand $G_\nabla(x,y)$.
	Furthermore, $G_\nabla(x,y)$ has Weyl filtration dimension $\ell(x)+\ell(y)$
	and belongs to the linkage class of $\omega_{xy}\Cdot0$.
	In the following, we will mostly restrict our attention to only one of the classes of modules $G_\Delta(x,y)$ or $G_\nabla(x,y)$, which is justified by the fact that
	$ G_\nabla(x,y) \cong G_\Delta(x,y)^\tau . $ Indeed, $G_\Delta(x,y)^\tau$ is a direct summand of
	\[ \big( \Delta(x\Cdot0) \otimes \Delta(y\Cdot0) \big)^\tau \cong \nabla(x\Cdot0) \otimes \nabla(y\Cdot0) , \]
	and $G_\Delta(x,y)^\tau$ is regular since $C_\mathrm{min}\big( G_\Delta(x,y)^\tau \big)_i \cong C_\mathrm{min}\big( G_\Delta(x,y) \big)_{-i}$ for all $i\in\Z$.
\end{Remark}

\begin{Definition}
	For $x,y\in W_\mathrm{ext}^+$, we call the $\G$-module $G_\Delta(x,y)$ from Theorem~\ref{thm:genericdirectsummandWeylmodule} the \emph{generic direct summand} of $\Delta(x\Cdot0) \otimes \Delta(y\Cdot0)$.
	Analogously, we call the $\G$-module $G_\nabla(x,y)$ from Remark \ref{rem:genericdirectsummandinducedmodule} the \emph{generic direct summand} of $\nabla(x\Cdot0) \otimes \nabla(y\Cdot0)$.
\end{Definition}

\begin{Remark} \label{rem:genericdirectsummand}
	The term \emph{generic direct summand} is justified by the fact that translates of $G_\Delta(x,y)$ appear \emph{generically} in Krull-Schmidt decompositions of tensor products of Weyl modules with highest weights in the alcoves $x\Cdot C_\mathrm{fund}$ and $y\Cdot C_\mathrm{fund}$.
	Indeed, for $\lambda,\mu \in C_\mathrm{fund}$ and $x,y\in W_\mathrm{aff}^+$, we have
	\begin{align*}
	\big( \Delta(x\Cdot\lambda) \otimes \Delta(y\Cdot\mu) \big)_\mathrm{reg} & \cong \big( T_0^\lambda \Delta(x\Cdot0) \otimes T_0^\mu \Delta(y\Cdot0) \big)_\mathrm{reg} \\
	& \cong \bigoplus_{\nu \in C_\mathrm{fund} \cap X} T_0^\nu \big( \Delta(x\Cdot0) \otimes \Delta(y\Cdot0) \big)_\mathrm{reg}^{\oplus c_{\lambda,\mu}^\nu} \\
	& \cong \bigoplus_{\nu \in C_\mathrm{fund} \cap X} T_0^\nu G_\Delta(x,y) ^{\oplus c_{\lambda,\mu}^\nu}
	\end{align*}
	by Theorems \ref{thm:translationtensor} and \ref{thm:genericdirectsummandWeylmodule}.
\end{Remark}

The following elementary lemma is an immediate consequence of Lemma \ref{lem:translationtensorfundamentalgroup}.

\begin{Lemma} \label{lem:genericdirectsummandWeylmodulefundamentalgroup}
	Let $x,y\in W_\mathrm{ext}^+$ and $\omega , \omega^\prime \in \Omega$. Then
	\[ G_\Delta(x\omega,y\omega^\prime) \cong T^{\omega \omega^\prime} G_\Delta(x,y) \qquad \text{and} \qquad G_\nabla(x\omega,y\omega^\prime) \cong T^{\omega \omega^\prime} G_\nabla(x,y) . \]
\end{Lemma}
\begin{proof}
	First note that $\Delta(x\omega\Cdot0) \cong T^\omega \Delta(x\Cdot0)$ and $\Delta(y\omega^\prime\Cdot0) \cong T^{\omega^\prime} \Delta(y\Cdot0)$.
	By Theorem \ref{thm:genericdirectsummandWeylmodule} and Lemma \ref{lem:translationtensorfundamentalgroup}, we have
	\begin{align*}
		G_\Delta(x\omega,y\omega^\prime) & \cong \big( \Delta(x\omega\Cdot0) \otimes \Delta(y\omega^\prime\Cdot0) \big)_\mathrm{reg} \\
		& \cong \big( T^\omega \Delta(x\Cdot0) \otimes T^{\omega^\prime} \Delta(y\Cdot0) \big)_\mathrm{reg} \\
		& \cong T^{\omega\omega^\prime} \big( \Delta(x\Cdot0) \otimes \Delta(y\Cdot0) \big)_\mathrm{reg} \\
		& \cong T^{\omega\omega^\prime} G_\Delta(x,y) ,
	\end{align*}
	as claimed.
	The isomorphism $G_\nabla(x\omega,y\omega^\prime) \cong T^{\omega \omega^\prime} G_\nabla(x,y)$ can be proven analogously.
\end{proof}

The regular indecomposable direct summands of tensor products of simple $\G$-modules in the extended principal block are in general not unique.
To get uniqueness, we need to impose a condition on the good filtration dimension.

\begin{Theorem} \label{thm:genericdirectsummandsimplemodule}
	Let $x,y\in W_\mathrm{ext}^+$. Then the tensor product $L(x\Cdot0) \otimes L(y\Cdot0)$ has a unique regular indecomposable direct summand $G(x,y)$ with $\gfd\big( G(x,y) \big) = \ell(x)+\ell(y)$.
\end{Theorem}
\begin{proof}
	First suppose that $x,y \in W_\mathrm{aff}^+$.
	Recall from Proposition \ref{prop:minimalcomplexsimplemodule} that the terms $C_\mathrm{min}\big( L(x\Cdot0) \big)_i$ of the minimal tilting complex of $L(x\Cdot0)$ are zero for $i>\ell(x)$ and negligible for $i=\ell(y)-1$,
	and similarly $C_\mathrm{min}\big( L(y\Cdot0) \big)_i$ is zero for $i>\ell(y)$ and negligible for $i=\ell(x)-1$.
	Furthermore, we have
	\[ C_\mathrm{min}\big( L(x\Cdot0) \big)_{\ell(x)} \cong T(0) \qquad \text{and} \qquad C_\mathrm{min}\big( L(y\Cdot0) \big)_{\ell(y)} \cong T(0) . \]
	Hence the degree $\ell(x)+\ell(y)-1$ term of the tensor product complex $C_\mathrm{min}\big( L(x\Cdot0) \big) \otimes C_\mathrm{min}\big( L(y\Cdot0) \big)$ is the negligible tilting module
	\begin{align*}
	& \Big( C_\mathrm{min}\big( L(x\Cdot0) \big) \otimes C_\mathrm{min}\big( L(y\Cdot0) \big) \Big)_{\ell(x)+\ell(y)-1} \\
	& \qquad = \Big( C_\mathrm{min}\big( L(x\Cdot0) \big)_{\ell(x)-1} \otimes C_\mathrm{min}\big( L(y\Cdot0) \big)_{\ell(y)} \Big) \oplus \Big( C_\mathrm{min}\big( L(x\Cdot0) \big)_{\ell(x)} \otimes C_\mathrm{min}\big( L(y\Cdot0) \big)_{\ell(y)-1} \Big) ,
	\end{align*}
	the degree $\ell(x)+\ell(y)$ term of $C_\mathrm{min}\big( L(x\Cdot0) \big) \otimes C_\mathrm{min}\big( L(y\Cdot0) \big)$ is
	\[ C_\mathrm{min}\big( L(x\Cdot0) \big)_{\ell(x)} \otimes C_\mathrm{min}\big( L(y\Cdot0) \big)_{\ell(y)} \cong T(0) \otimes T(0) \cong T(0) , \]
	and the terms of $C_\mathrm{min}\big( L(x\Cdot0) \big) \otimes C_\mathrm{min}\big( L(y\Cdot0) \big)$ in degree $i>\ell(x)+\ell(y)$ are all zero.
	Arguing as in the proof of Theorem \ref{thm:genericdirectsummandWeylmodule}, we see that
	\[ C_\mathrm{min}\big( L(x\Cdot0) \otimes L(y\Cdot0) \big)_{\ell(x)+\ell(y)} \cong T(0) , \]
	whereas $C_\mathrm{min}\big( L(x\Cdot0) \otimes L(y\Cdot0) \big)_i = 0$ for $i>\ell(x)+\ell(y)$.
	This implies that there is a unique indecomposable direct summand $G(x,y)$ of $L(x\Cdot0) \otimes L(y\Cdot0)$ whose minimal tilting complex has a non-zero term in degree $\ell(x)+\ell(y)$, and the latter satisfies
	\[ C_\mathrm{min}\big( G(x,y) \big)_{\ell(x)+\ell(y)} \cong T(0) . \]
	In particular, $G(x,y)$ is the unique indecomposable direct summand of $L(x\Cdot0) \otimes L(y\Cdot0)$ with good filtration dimension $\ell(x)+\ell(y)$ by \eqref{eq:gfdwfd}, and $G(x,y)$ is regular.
	
For arbitrary $x,y \in W_\mathrm{ext}^+$, we can write $x = x_0 \omega_x$ and $y = y_0 \omega_y$ with $x_0,y_0 \in W_\mathrm{aff}^+$, and the claim follows with $G(x,y) \cong T^{\omega_{xy}} G(x_0,y_0)$ by Lemma \ref{lem:translationtensorfundamentalgroup} (as in the proof of Theorem \ref{thm:genericdirectsummandWeylmodule}).
\end{proof}

\begin{Remark} \label{rem:genericdirectsummandsimplemodulelastterm}
	Observe that the proof of Theorem \ref{thm:genericdirectsummandsimplemodule} implies that
	\[ C_\mathrm{min}\big( G(x,y) \big)_{\ell(x)+\ell(y)} \cong T(\omega_{xy}\Cdot0) \]
	for all $x,y\in W_\mathrm{ext}^+$, and that $G(x,y)$ belongs to the linkage class of $\omega_{xy}\Cdot0$.
	Furthermore, if $x,y \in W_\mathrm{aff}^+$ then $G(x,y)$ is the unique indecomposable direct summand of the tensor product $L(x\Cdot0) \otimes L(y\Cdot0)$ with good filtration dimension $\ell(x)+\ell(y)$.
\end{Remark}

\begin{Remark}
	For $x,y\in W_\mathrm{ext}^+$, one can show as in Theorem \ref{thm:genericdirectsummandsimplemodule} that $L(x\Cdot0) \otimes L(y\Cdot0)$ has a unique regular indecomposable direct summand $G^\prime(x,y)$ with $\wfd\big( G^\prime(x,y) \big) = \ell(x)+\ell(y)$. In the following, we will only study the modules $G(x,y)$, which is justified by the fact that
	$ G^\prime(x,y) \cong G(x,y)^\tau . $ Indeed, $G(x,y)^\tau$ is a direct summand of $\big( L(x\Cdot0) \otimes L(y\Cdot0) \big)^\tau \cong L(x\Cdot0) \otimes L(y\Cdot0)$ with
	\[ \wfd\big( G(x,y)^\tau \big) = \gfd\big( G(x,y) \big) = \ell(x)+\ell(y) , \]
	and $G(x,y)^\tau$ is regular since $C_\mathrm{min}\big( G(x,y)^\tau \big)_i \cong C_\mathrm{min}\big( G(x,y) \big)_{-i}$ for all $i\in\Z$.
\end{Remark}

\begin{Definition}
	For $x,y\in W_\mathrm{ext}^+$, we call the indecomposable $\G$-module $G(x,y)$ from Theorem~\ref{thm:genericdirectsummandsimplemodule} the \emph{generic direct summand} of $L(x\Cdot0) \otimes L(y\Cdot0)$.
\end{Definition}

Lemma \ref{lem:genericdirectsummandWeylmodulefundamentalgroup} has an obvious analogue for generic direct summands of tensor products of simple $\G$-modules.

\begin{Lemma} \label{lem:genericdirectsummandsimplemodulefundamentalgroup}
	Let $x,y\in W_\mathrm{ext}^+$ and $\omega , \omega^\prime \in \Omega$. Then
	\[ G(x\omega,y\omega^\prime) \cong T^{\omega \omega^\prime} G(x,y) . \]
\end{Lemma}
\begin{proof}
	We can essentially copy the proof of Lemma \ref{lem:genericdirectsummandWeylmodulefundamentalgroup}, replacing Weyl modules by simple modules.
	The only additional fact that one needs to use is that the translation functor $T^{\omega\omega^\prime}$ preserves the good filtration dimension.
\end{proof}

\begin{Remark} \label{rem:genericdirectsummandstronglyregular}
	A non-zero $\G$-module $M$ with $\gfd(M) = d$ is called \emph{strongly regular} if $C_\mathrm{min}(M)_{d-1}$ is negligible and $C_\mathrm{min}(M)_d$ is not \cite[Definition 4.1]{GruberLinkageTranslation}.
	By the proofs of Theorems \ref{thm:genericdirectsummandWeylmodule} and \ref{thm:genericdirectsummandsimplemodule} and by Remarks \ref{rem:genericdirectsummandWeylmodulelastterm} and \ref{rem:genericdirectsummandsimplemodulelastterm}, the $\G$-modules $G(x,y)$ and $G_\Delta(x,y)$ are strongly regular for all $x,y\in W_\mathrm{ext}^+$.
\end{Remark}

\begin{Remark} \label{rem:genericdirectsummanddifferentmodules}
	In principle, there is no reason why one would need to restrict one's attention to the study of regular indecomposable direct summands of tensor products of $\G$-modules that belong to the same class of modules (such as Weyl modules or simple modules).
	For $x,y\in W_\mathrm{ext}^+$, the proof of Theorem \ref{thm:genericdirectsummandWeylmodule} can easily be adapted to show that a tensor product of the form $\Delta(x\Cdot0) \otimes \nabla(y\Cdot0)$ has a unique regular indecomposable direct summand $G_{\Delta,\nabla}(x,y)$, and that $T(\omega_{xy}\Cdot0)$ appears with multiplicity one in a Krull-Schmidt decomposition of $C_\mathrm{min}\big( G_{\Delta,\nabla}(x,y) \big)_{\ell(x)-\ell(y)}$.
	
	Similarly, we can adapt the proof of Theorem \ref{thm:genericdirectsummandsimplemodule} to show that $L(x\Cdot0) \otimes \Delta(y\Cdot0)$ has a unique regular indecomposable direct summand with good filtration dimension $\ell(x)+\ell(y)$, and that $L(x\Cdot0) \otimes \nabla(y\Cdot0)$ has a unique regular indecomposable direct summand with Weyl filtration dimension $\ell(x)+\ell(y)$.
\end{Remark}

\begin{Remark} \label{rem:genericdirectsummanditerated}
	For the most part of this manuscript, we have been (and will be) restricting our attention to tensor products of two $\G$-modules, but one may also ask about regular indecomposable direct summands of iterated tensor products with more than two constituents.
	For $x_1,\ldots,x_n \in W_\mathrm{ext}^+$,
	one can use the techniques from the proofs of Theorems \ref{thm:genericdirectsummandWeylmodule} and \ref{thm:genericdirectsummandsimplemodule} to show that the iterated tensor product
	\[ \Delta(x_1\Cdot0) \otimes \cdots \otimes \Delta(x_n\Cdot0) \]
	has a unique regular indecomposable direct summand $G_\Delta(x_1,\ldots,x_n)$ (which has good filtration dimension $\ell(x_1) + \cdots + \ell(x_n)$) and that the iterated tensor product
	\[ L(x_1\Cdot0) \otimes \cdots \otimes L(x_n\Cdot0) \] 
	has a unique regular indecomposable direct summand $G(x_1,\ldots,x_n)$ that satisfies
	\[ \gfd\big( G(x_1,\ldots,x_n) \big) = \ell(x_1)+\cdots+\ell(x_n) . \]
\end{Remark}

\section{The Steinberg-Lusztig tensor product theorem} \label{sec:SteinbergLusztigTPtheorem}

Let us write $X_1 = \{ \lambda \in X^+ \mid (\lambda,\alpha^\vee) < \ell \text{ for all } \alpha \in \Pi \}$ for the set of $\ell$-restricted weights.
Then for all $\lambda \in X^+$, there are uniquely determined weights $\lambda_0 \in X_1$ and $\lambda_1 \in X^+$ such that $\lambda = \lambda_0 + \ell \lambda_1$, and by the Steinberg-Lusztig tensor product theorem (see Sections II.3.16 and II.H.10 in \cite{Jantzen}), the simple $\G$-module $L(\lambda)$ admits a tensor product decomposition $L(\lambda) \cong L(\lambda_0) \otimes L(\ell\lambda_1)$.
Furthermore, the simple $\G$-module $L(\ell\lambda_1)$ can be described as a Frobenius twist of a simple module of highest weight $\lambda_1$ for $\G$ (in the modular case) or for $\mathfrak{g}$ (in the quantum case).
Our aim in this section is to establish an analogue of the Steinberg-Lusztig tensor product theorem for generic direct summands of tensor products of simple $\G$-modules.
We start by recalling some results about Frobenius kernels, small quantum groups and Frobenius morphisms, for which it will be necessary to distinguish the modular case and the quantum case.

\subsection*{The modular case}

The group scheme $\G$ admits a \emph{Frobenius endomorphism} $\mathrm{Fr}\colon \G \to \G$ that fixes the maximal torus and Borel subgroup corresponding to $\Phi^+$; see Section II.3.1 in \cite{Jantzen}.
The \emph{Frobenius kernels} $\G_r \coloneqq \ker(\mathrm{Fr}^r)$ for $r>0$ are infinitesimal subgroup schemes of $\G$ (in the sense of Section I.8.1 in \cite{Jantzen}).
The Frobenius twist $M^{[r]}$ of a $\G$-module $M$ is defined by composing the action of $\G$ on $M$ with the $r$-th power $\mathrm{Fr}^r$ of the Frobenius endomorphism.
Note that the restriction to $\G_r$ of the Frobenius twist $M^{[r]}$ is a direct sum of copies of the trivial one-dimensional $\G_r$-module.
Conversely, if $N$ is a $\G$-module whose restriction to $\G_r$ is a direct sum of copies of the trivial one-dimensional $\G_r$-module then there exists a $\G$-module $M$, uniquely determined by $N$, with $N = M^{[r]}$, and we write $M = N^{[-r]}$.
Steinberg's tensor product theorem asserts that for $\mu \in X_1$ and $\lambda \in X$, we have $L(\mu + \ell\lambda) \cong L(\mu) \otimes L(\lambda)^{[1]}$.
For our applications to generic direct summands, we will need an indecomposability criterion for twisted tensor products of the form $M \otimes N^{[1]}$, where $M$ and $N$ are indecomposable $\G$-modules.
The following result of S.\ Donkin will be very useful; see the lemma in Section 2 of \cite{DonkinQuestionVerma}.

\begin{Lemma} \label{lem:Donkintensorproductindecomposable}
	Let $V$ and $W$ be $\G$-modules such that $V$ is indecomposable as a $\G_1$-module, $W$ is indecomposable as a $\G_r$-module for some $r > 0$ and the restriction to $\G_1$ of $W$ is a direct sum of copies of the trivial one-dimensional $\G_1$-module.
	Then $V \otimes W$ is indecomposable as a $\G_r$-module.
\end{Lemma}

Note that for every indecomposable $\G$-module $M$, we can choose $r>0$ such that $M$ is also indecomposable as a $\G_r$-module, because $\End_{\G}(M) = \End_{\G_r}(M)$ when $r$ is large enough by point (6) in Section I.9.8 in \cite{Jantzen}.
Conversely, a $\G$-module which is indecomposable as a $\G_r$-module for some $r>0$ is also indecomposable as a $\G$-module.
By applying Lemma \ref{lem:Donkintensorproductindecomposable} to the indecomposable $\G$-module $M^{[1]}$, we obtain the following results:

\begin{Corollary} \label{cor:twistedtensorproductindecomposable}
	Let $V$ and $M$ be $\G$-modules such that $V$ is indecomposable as a $\G_1$-module and $M$ is indecomposable as a $\G$-module. Then $V \otimes M^{[1]}$ is indecomposable as a $\G$-module.
\end{Corollary}

\begin{Corollary} \label{cor:twistedtensorproductindecomposableiterated}
	Let $M_0,\ldots,M_r$ be $\G$-modules such that $M_0,\ldots,M_{r-1}$ are indecomposable as $\G_1$-modules and $M_r$ is indecomposable as a $\G$-module.
	Then the tensor product $M_0 \otimes M_1^{[1]} \otimes \cdots \otimes M_r^{[r]}$ is indecomposable as a $\G$-module.
\end{Corollary}
\begin{proof}
	Note that $M_0 \otimes M_1^{[1]} \otimes \cdots \otimes M^{[r]} \cong M_0 \otimes \big( M_1 \otimes \cdots \otimes M_r^{[r-1]} \big)^{[1]}$.
	Using this observation, the claim follows from Corollary \ref{cor:twistedtensorproductindecomposable}, by induction on $r$.
\end{proof}

\subsection*{The quantum case}

The quantum Frobenius morphism, constructed by G.~Lusztig in \cite{LusztigTensorProduct}, is a surjective Hopf algebra homomorphism $\mathrm{Fr} \colon \G = U_\zeta(\mathfrak{g}) \to U(\mathfrak{g})$, where $U(\mathfrak{g})$ denotes the universal enveloping algebra of $\mathfrak{g}$.
It gives rise to an exact and monoidal Frobenius twist functor $M \mapsto M^{[1]}$ from $\Rep(\mathfrak{g})$ to $\Rep(\G)$.
Let us write $L_\C(\lambda)$ for the simple $\mathfrak{g}$-module of highest weight $\lambda \in X^+$.
Then, for $\mu \in X_1$ and $\lambda \in X^+$, we have $L(\mu+\ell\lambda) \cong L(\mu) \otimes L_\C(\lambda)^{[1]}$ by Lusztig's tensor product theorem; see Section II.H.10 in \cite{Jantzen}.

Let us write $\G_1 = u_\zeta(\mathfrak{g})$ for the \emph{small quantum group}, defined as in Section II.H.7 in \cite{Jantzen}.
It is a finite-dimensional normal Hopf-subalgebra of $\G = U_\zeta(\mathfrak{g})$ and plays a role analogous to that of the first Frobenius kernel in the modular case.
The restriction to $\G_1$ of a simple $\G$-module $L(\lambda)$ with $\lambda \in X_1$ is also simple as a $\G_1$-module, and this provides a complete set of representatives for the isomorphism classes of simple $\G_1$-modules \cite[Section II.H.13]{Jantzen}.
The restriction to $\G_1$ of the Frobenius twist $M^{[1]}$ of a $\G$-module $M$ is a direct sum of copies of the trivial one-dimensional $\G_1$-module.

Next we establish a quantum analogue of the indecomposability criterion for twisted tensor products from Corollary \ref{cor:twistedtensorproductindecomposable}.
Note that the hypotheses in the following lemma are stronger than those that we imposed in the modular case.
We do not know if a direct analogue of Corollary \ref{cor:twistedtensorproductindecomposable} holds in the quantum case.

\begin{Lemma} \label{lem:twistedtensorproductindecomposablequantum}
	Let $V$ be a $\G$-module that has simple socle as a $\G_1$-module, and let $L$ be a simple $\mathfrak{g}$-module.
	Then $V \otimes L^{[1]}$ is an indecomposable $\G$-module.
\end{Lemma}
\begin{proof}
	By the assumption, there exists a weight $\lambda \in X_1$ such that $\dim \Hom_{\G_1}\big( L(\lambda) , V \big) = 1$, and as $\G_1$ acts trivially on $L^{[1]}$, there are isomorphisms of $\G$-modules
	\[ \Hom_{\G_1}\big( L(\lambda) , V \otimes L^{[1]} \big) \cong \Hom_{\G_1}\big( L(\lambda) , V \big) \otimes L^{[1]} \cong L^{[1]} . \]
	Suppose for a contradiction that there is a non-trivial direct sum decomposition $V \otimes L^{[1]} \cong M_1 \oplus M_2$.
	As $\G_1$-modules, both $M_1$ and $M_2$ are isomorphic to (non-empty) direct sums of copies of $V$, so we obtain a non-trivial direct sum decomposition (as $\G$-modules)
	\[ L^{[1]} \cong \Hom_{\G_1}\big( L(\lambda) , V \otimes L^{[1]} \big) \cong \Hom_{\G_1}\big( L(\lambda) , M_1 \big) \oplus \Hom_{\G_1}\big( L(\lambda) , M_2 \big) , \]
	contradicting the simplicity of $L^{[1]}$.
\end{proof}

\begin{Remark}
	As explained in Section 3.4 in \cite{AndersenPoloWenInjective}, the $\G_1$-socle of a $\G$-module coincides with its $\G$-socle.
	Therefore, the condition that $V$ is a $\G$-module with simple socle as a $\G_1$-module can equivalently be stated as $V$ having simple socle with $\ell$-restricted highest weight as a $\G$-module.
\end{Remark}
\bigskip

Now let us return to the general case.
Before we discuss generic direct summands, we need to establish some properties of the length function on $W_\mathrm{ext}$.
Recall that $\ell(x)$ is defined as the number of reflection hyperplanes $H_{\beta,r}$ that separate $C_\mathrm{fund}$ and $x \Cdot C_\mathrm{fund}$, for $x \in W_\mathrm{ext}$, $\beta \in \Phi^+$ and $r \in \Z$.
We define $\rho^\vee$ to be the half-sum of all positive coroots, that is
\[ \rho^\vee \coloneqq \frac12 \cdot \sum_{ \beta \in \Phi^+ } \beta^\vee . \]

\begin{Lemma} \label{lem:lengthdominant}
	Let $x \in W_\mathrm{ext}^+$ and $\lambda \in X^+$.
	Then $\ell(t_\lambda x) = \ell(t_\lambda) + \ell(x)$ and $\ell(t_\lambda) = 2 \cdot (\lambda,\rho^\vee)$.
\end{Lemma}
\begin{proof}
	For $\beta \in \Phi^+$ and $r \in \Z$, the reflection hyperplane $H_{\beta,r}$ separates $C_\mathrm{fund}$ and $x\Cdot C_\mathrm{fund}$ if and only if $0 < \ell r < ( x\Cdot0 + \rho , \beta^\vee )$, because $x\Cdot0$ is dominant.
	Analogously, $H_{\beta,r}$ separates $C_\mathrm{fund}$ and $t_\lambda x\Cdot C_\mathrm{fund}$ if and only if $0 < \ell r < ( t_\lambda x\Cdot0 + \rho , \beta^\vee ) = \ell \cdot (\lambda,\beta^\vee) + ( x\Cdot0 + \rho , \beta^\vee )$, and $H_{\beta,r}$ separates $C_\mathrm{fund}$ and $t_\lambda \Cdot C_\mathrm{fund}$ if and only if $0 < \ell r < ( t_\lambda \Cdot0 + \rho , \beta^\vee ) = \ell \cdot (\lambda,\beta^\vee) + (\rho,\beta^\vee)$.
	As $0 < (\rho,\beta^\vee) < \ell$ for all $\beta \in \Phi^+$ (because $\ell \geq h$), it follows that
	\[ \ell(t_\lambda x) - \ell(x) = \sum_{\beta\in\Phi^+} (\lambda,\beta^\vee) = \ell(t_\lambda) , \]
	and $\ell(t_\lambda) = 2 \cdot (\lambda,\rho^\vee)$, as claimed.
\end{proof}

Now we are ready to prove our results relating generic direct summands with the tensor product theorems of R.~Steinberg and G.~Lusztig.
We fix two elements $x,y\in W_\mathrm{ext}^+$ and write
\[ x\Cdot 0 = \lambda^\prime + \ell \lambda \qquad \text{and} \qquad y \Cdot0 = \mu^\prime + \ell \mu \]
with $\lambda^\prime,\mu^\prime\in X_1$ and $\lambda,\mu \in X^+$.
Furthermore, we set 
\[ x_0 \coloneqq t_{-\lambda} x \qquad \text{and} \qquad y_0 = t_{-\mu} y . \]
Observe that $x_0\Cdot0 = x\Cdot0 - \ell\lambda = \lambda^\prime$, whence $x_0\in W_\mathrm{ext}^+$ and
\[ \ell(x) = \ell(t_\lambda x_0) = \ell(t_\lambda) + \ell(x_0) \]
by Lemma \ref{lem:lengthdominant}, and similarly $\ell(y) = \ell(t_\mu) + \ell(y_0)$.

\begin{Theorem} \label{thm:lusztigtensorproductgenericdirectsummand}
	Suppose that we are in the quantum case.
	Then $G(x,y)$ is a direct summand of the tensor product $G(x_0,y_0) \otimes L_\C(\lambda+\mu)^{[1]}$.
	If $G(x_0,y_0)$ has simple socle as a $\G_1$-module then
	\[ G(x,y) \cong G(x_0,y_0) \otimes L_\C(\lambda+\mu)^{[1]} . \]
\end{Theorem}
\begin{proof}
	Recall from Remark \ref{rem:genericdirectsummandstronglyregular} that the generic direct summand $G(x_0,y_0)$ is strongly regular of good filtration dimension $\ell(x_0)+\ell(y_0)$, and that $L_\C(\lambda+\mu)^{[1]} \cong L(t_{\lambda+\mu}\Cdot0) \cong G(t_{\lambda+\mu},e)$ is strongly regular of good filtration dimension $\ell(t_{\lambda+\mu})$.
	The tensor product $G(x_0,y_0) \otimes L_\C(\lambda+\mu)^{[1]}$ is strongly regular of good filtration dimension
	\[ \ell(x_0)+\ell(y_0) + \ell(t_{\lambda+\mu}) = \ell(x) + \ell(y) \]
	by Lemma \ref{lem:lengthdominant} and \cite[Lemma 4.3]{GruberLinkageTranslation}.
	Thus, $G(x_0,y_0) \otimes L_\C(\lambda+\mu)^{[1]}$ has a regular indecomposable direct summand $M$ with good filtration dimension $\ell(x) + \ell(y)$.
	Furthermore, $G(x_0,y_0) \otimes L_\C(\lambda+\mu)^{[1]}$ is a direct summand of the tensor product
	\[ \big( L(x_0\Cdot0) \otimes L(y_0\Cdot0) \big) \otimes \big( L_\C(\lambda) \otimes L_\C(\mu) \big)^{[1]} \cong L(x\Cdot0) \otimes L(y\Cdot0) , \]
	and it follows that $M$ is a direct summand of $L(x\Cdot0) \otimes L(y\Cdot0)$.
	As $G(x,y)$ is the unique regular indecomposable direct summand of $L(x\Cdot0) \otimes L(y\Cdot0)$ with good filtration dimension $\ell(x) + \ell(y)$, we conclude that $G(x,y) \cong M$ is a direct summand of $G(x_0,y_0) \otimes L_\C(\lambda+\mu)^{[1]}$.
	If $G(x_0,y_0)$ has simple socle as a $\G_1$-module then $G(x_0,y_0) \otimes L_\C(\lambda+\mu)^{[1]}$ is indecomposable by Lemma \ref{lem:twistedtensorproductindecomposablequantum}, and so
	\[ G(x,y) \cong G(x_0,y_0) \otimes L_\C(\lambda+\mu)^{[1]} , \]
	as claimed.
\end{proof}

\begin{Remark} \label{rem:regularpartFrobeniusquantum}
	Suppose that we are in the quantum case.
	If all regular indecomposable direct summands of $L(x_0\Cdot0) \otimes L(y_0\Cdot0)$ have simple socles as $\G_1$-modules then we can give a complete description of $\big( L(x\Cdot0) \otimes L(y\Cdot0) \big)_\mathrm{reg}$.
	Concretely, let
	\[ \big( L(x_0\Cdot0) \otimes L(y_0\Cdot0) \big)_\mathrm{reg} \cong M_1 \oplus \cdots \oplus M_r , \]
	where $M_1,\ldots,M_r$ are regular and have simple socles as $\G_1$-modules, and write
	\[ L_\C(\lambda) \otimes L_\C(\mu) \cong \bigoplus_{\nu\in X^+} L_\C(\nu)^{\oplus d_{\lambda,\mu}^\nu} \]
	for certain multiplicities $d_{\lambda,\mu}^\nu \in \Z_{\geq 0}$.
	Note that $M \otimes L^{[1]}$ is regular for every regular $\G$-module $M$ and every simple $\mathfrak{g}$-module $L$, because the regular $\G$-module $M$ is a direct summand of $M \otimes L^{[1]} \otimes L^{*[1]}$.
	Arguing as in the proof of Theorem \ref{thm:lusztigtensorproductgenericdirectsummand}, we see that
	\[ \big( L(x\Cdot0) \otimes L(y\Cdot0) \big)_\mathrm{reg} \cong \bigoplus_{\nu \in X^+} \bigoplus_{i=1}^r \big( M_i \otimes L_\C(\nu)^{[1]} \big)^{\oplus d_{\lambda,\mu}^\nu} , \]
	where $M_i \otimes L_\C(\nu)^{[1]}$ is regular and indecomposable for all $\nu \in X^+$ and $i=1,\ldots,r$.
\end{Remark}

Our next goal is to establish a modular analogue of Theorem \ref{thm:lusztigtensorproductgenericdirectsummand}.
One key argument in the proof of the theorem was that the (quantum) Frobenius twist of a simple $\mathfrak{g}$-module is always strongly regular, because it is a simple $U_\zeta(\mathfrak{g})$-module of $\ell$-regular highest weight.
Since in the modular case, the domain of the Frobenius twist functor is the non-semisimple category $\Rep(\G)$, we will need to replace this argument by the following result:

\begin{Proposition} \label{prop:Frobeniustwistregular}
	Suppose that we are in the modular case.
	Let $M$ be a $\G$-module in $\Rep_\nu(\G)$, for some $\nu \in \overline{C}_\mathrm{fund} \cap X$.
	Then the Frobenius twist $M^{[1]}$ is strongly regular.
	More precisely, define
	\[ d = d_M \coloneqq \max\big\{ 2 \cdot (\gamma,\rho^\vee) \mathrel{\big|} \gamma\in X^+ \text{ with } [M:L(\gamma)] \neq 0 \big\} . \]
	Then
	\begin{enumerate}
	 \item $C_\mathrm{min}\big(M^{[1]}\big)_i = 0$ for $i > d$;
	 \item $C_\mathrm{min}\big(M^{[1]}\big)_{d-1}$ is negligible;
	 \item $C_\mathrm{min}\big(M^{[1]}\big)_d \cong T(\omega_{t_\nu}\Cdot0)^{\oplus r}$ for some $r>0$.
	\end{enumerate}
\end{Proposition}
\begin{proof}
	Note that the composition factors of $M^{[1]}$ are of the form $L(\gamma^\prime)^{[1]} \cong L(\ell\gamma^\prime)$, for $\gamma^\prime \in X^+$ such that $[M:L(\gamma^\prime)] \neq 0$.
	The good filtration dimension of $L(\ell\gamma^\prime) = L(t_{\gamma^\prime}\Cdot0)$ is $\ell(t_{\gamma^\prime}) = 2 \cdot (\gamma^\prime,\rho^\vee)$ by Lemma \ref{lem:lengthdominant}, and using Lemma 1.14 in \cite{GruberMinimalTilting}, it easily follows that $\gfd\big( M^{[1]} \big) \leq d$.
	By \eqref{eq:gfdwfd}, this also implies that $C_\mathrm{min}\big( M^{[1]} \big)_i = 0$ for all $i>d$.
	We prove the remaining claims by induction on the composition length of $M$.
	
	If $M\cong L(\gamma)$ is simple then $M^{[1]} \cong L(\ell\gamma) = L(t_\gamma\Cdot0)$ and $\ell(t_\gamma) = 2 \cdot (\gamma,\rho^\vee) = d$ by Lemma \ref{lem:lengthdominant}.
	Furthermore, we have $\omega_{t_\gamma} = \omega_{t_\nu}$ because the weights $\gamma$ and $\nu$ belong to the same $\Z\Phi$-coset in $X$, so the claim follows from Proposition \ref{prop:minimalcomplexsimplemodule}.
	Now suppose that $M$ has at least two composition factors and fix a short exact sequence
	\[ 0 \to L \to M \to N \to 0 \]
	with $L$ and $N$ non-zero.
	By induction, we may assume that the proposition holds for $L$ and $N$.
	
	Note that by assumption, both $L$ and $N$ belong to the linkage class $\Rep_\nu(\G)$.
	In particular, all highest weights of composition factors of $L$ and of $N$ belong to the same $\Z\Phi$-coset in $X$.
	As $(\beta,\rho^\vee) \in \Z$ for all $\beta\in\Phi$, it follows that $d_L$ and $d_N$ have the same parity.
	Furthermore, by applying Lemma 2.17 in \cite{GruberMinimalTilting} to the short exact sequence
	\[ 0 \to L^{[1]} \to M^{[1]} \to N^{[1]} \to 0 , \]
	we see that there are split embeddings
	\[ C_\mathrm{min}\big( M^{[1]} \big)_i \stackrel{\oplus}{\subseteq} C_\mathrm{min}\big( L^{[1]} \big)_i \oplus C_\mathrm{min}\big( N^{[1]} \big)_i \eqqcolon C_i \]
	for all $i \in \Z$.
	It is straightforward to see that $d = d_M = \max\{ d_L, d_N \}$, and we distinguish three cases:
	\begin{enumerate}
	\item Suppose that $d=d_L > d_N$.
	Then $d_N \leq d_L-2 = d-2$ because $d_L$ and $d_N$ have the same parity, and it follows that $C_\mathrm{min}\big( N^{[1]} \big)_i = 0$ for $i>d-2$.
	In particular, we have
	\[ C_\mathrm{min}\big( M^{[1]} \big)_{d-1} \stackrel{\oplus}{\subseteq} C_\mathrm{min}\big( L^{[1]} \big)_{d-1} \qquad \text{and} \qquad C_\mathrm{min}\big( M^{[1]} \big)_d \stackrel{\oplus}{\subseteq} C_\mathrm{min}\big( L^{[1]} \big)_d \cong T(\omega_{t_\nu}\Cdot0)^{\oplus r} \]
	for some $r>0$, whence $C_\mathrm{min}\big( M^{[1]} \big)_{d-1}$ is negligible.
	Furthermore, Lemma 2.17 in \cite{GruberMinimalTilting} implies that
	\[ C_\mathrm{min}\big(M^{[1]}\big)_d \cong C_\mathrm{min}\big(L^{[1]}\big)_d \cong T(\omega_{t_\gamma} \Cdot 0)^{\oplus r} , \]
	because $C_{d-1} = C_\mathrm{min}\big( L^{[1]} \big)_{d-1}$ is negligible and $C_{d+1}=0$.
	\item Suppose that $d=d_N > d_L$.
	Then
	$d_L \leq d_N-2 = d-2$ because $d_L$ and $d_N$ have the same parity.
	The claim follows precisely as in case (1), with the roles of $L$ and $N$ interchanged.
	\item Suppose that $d = d_L = d_N$.
	Then $C_\mathrm{min}\big( M^{[1]} \big)_{d-1}$ is negligible because
	\[ C_\mathrm{min}\big( M^{[1]} \big)_{d-1} \stackrel{\oplus}{\subseteq} C_\mathrm{min}\big( L^{[1]} \big)_{d-1} \oplus C_\mathrm{min}\big( N^{[1]} \big)_{d-1} , \]
	and again using Lemma 2.17 in \cite{GruberMinimalTilting}, we see that
	\[ C_\mathrm{min}\big( M^{[1]} \big)_d \cong C_\mathrm{min}\big( L^{[1]} \big)_d \oplus C_\mathrm{min}\big( N^{[1]} \big)_d \cong T(\omega_{t_\nu}\Cdot0)^{\oplus r} \]
	for some $r>0$, because $C_{d-1}$ is negligible and $C_{d+1}=0$. \qedhere
	\end{enumerate}
\end{proof}

Now we are ready to prove a modular analogue of Theorem \ref{thm:lusztigtensorproductgenericdirectsummand}.
Recall that we fix $x,y \in W_\mathrm{ext}^+$ and write
$x\Cdot 0 = \lambda^\prime + \ell \lambda$ and $y \Cdot0 = \mu^\prime + \ell \mu $
with $\lambda^\prime,\mu^\prime\in X_1$ and $\lambda,\mu \in X^+$. As before, we set
\[ x_0 \coloneqq t_{-\lambda} x \qquad \text{and} \qquad y_0 \coloneqq t_{-\mu} y . \]
\begin{Definition}
	Let $M(\lambda,\mu)$ be the unique indecomposable direct summand of $L(\lambda) \otimes L(\mu)$ that has the simple $\G$-module $L(\lambda+\mu)$ as a composition factor.
\end{Definition}
Equivalently, $M(\lambda,\mu)$ can be defined as the unique indecomposable direct summand of $L(\lambda) \otimes L(\mu)$ with a non-zero $(\lambda+\mu)$-weight space.

\begin{Theorem} \label{thm:steinbergtensorproductgenericdirectsummand}
	Suppose that we are in the modular case.
	Then $G(x,y)$ is a direct summand of the tensor product $G(x_0,y_0) \otimes M(\lambda,\mu)^{[1]}$.
	If $G(x_0,y_0)$ is indecomposable as a $\G_1$-module then
	\[ G(x,y) \cong G(x_0,y_0) \otimes M(\lambda,\mu)^{[1]} . \]
\end{Theorem}
\begin{proof}
	Recall from Remark \ref{rem:genericdirectsummandstronglyregular} that the generic direct summand $G(x_0,y_0)$ is strongly regular of good filtration dimension $\ell(x_0)+\ell(y_0)$.
	The $\G$-module $M(\lambda,\mu)^{[1]}$ is strongly regular of good filtration dimension $2 \cdot (\lambda+\mu,\rho^\vee)$ by Proposition \ref{prop:Frobeniustwistregular}, and using Lemma \ref{lem:lengthdominant} and \cite[Lemma 4.3]{GruberLinkageTranslation}, it follows that
	the tensor product $G(x_0,y_0) \otimes M(\lambda,\mu)^{[1]}$ is strongly regular of good filtration dimension
	\[ \ell(x_0)+\ell(y_0) + 2 \cdot (\lambda+\mu,\rho^\vee) = \ell(x) + \ell(y) . \]
	Thus, $G(x_0,y_0) \otimes M(\lambda,\mu)^{[1]}$ has a regular indecomposable direct summand $M$ with good filtration dimension $\ell(x) + \ell(y)$.
	Furthermore, $G(x_0,y_0) \otimes M(\lambda,\mu)^{[1]}$ is a direct summand of the tensor product
	\[ \big( L(x_0\Cdot0) \otimes L(y_0\Cdot0) \big) \otimes \big( L(\lambda) \otimes L(\mu) \big)^{[1]} \cong L(x\Cdot0) \otimes L(y\Cdot0) , \]
	and it follows that $M$ is a direct summand of $L(x\Cdot0) \otimes L(y\Cdot0)$.
	As $G(x,y)$ is the unique regular indecomposable direct summand of $L(x\Cdot0) \otimes L(y\Cdot0)$ with good filtration dimension $\ell(x) + \ell(y)$, we conclude that $G(x,y) \cong M$ is a direct summand of $G(x_0,y_0) \otimes M(\lambda,\mu)^{[1]}$.
	If $G(x_0,y_0)$ is indecomposable as a $\G_1$-module then $G(x_0,y_0) \otimes M(\lambda,\mu)^{[1]}$ is indecomposable by Corollary \ref{cor:twistedtensorproductindecomposable}, and so
	\[ G(x,y) \cong G(x_0,y_0) \otimes M(\lambda,\mu)^{[1]} , \]
	as claimed.
\end{proof}

\begin{Remark} \label{rem:regularpartFrobeniusmodular}
	Suppose that we are in the modular case.
	If all regular indecomposable direct summands of $L(x_0\Cdot0) \otimes L(y_0\Cdot0)$ are strongly regular and indecomposable as $\G_1$-modules then we can give a description of $\big( L(x\Cdot0) \otimes L(y\Cdot0) \big)_\mathrm{reg}$.
	Concretely, let
	\[ \big( L(x_0\Cdot0) \otimes L(y_0\Cdot0) \big)_\mathrm{reg} \cong M_1 \oplus \cdots \oplus M_r , \]
	where $M_1,\ldots,M_r$ are strongly regular and indecomposable as $\G_1$-modules, and fix a Krull-Schmidt decomposition
	\[ L(\lambda) \otimes L(\mu) \cong N_1 \oplus N_2 \oplus \cdots \oplus N_s . \]
	Then, arguing as in the proof of Theorem \ref{thm:lusztigtensorproductgenericdirectsummand}, we see that
	\[ \big( L(x\Cdot0) \otimes L(y\Cdot0) \big)_\mathrm{reg} \cong \bigoplus_{i=1}^r \bigoplus_{j=1}^s \big( M_i \otimes N_j^{[1]} \big) , \]
	where $M_i \otimes N_j^{[1]}$ is indecomposable and strongly regular for all $i=1,\ldots,r$ and $j=1,\ldots,s$.
\end{Remark}

In view of Theorem \ref{thm:steinbergtensorproductgenericdirectsummand}, it is important to determine the $\G$-modules $M(\lambda,\mu)$.
Note that the weights $\lambda$ and $\mu$ can be written as $\lambda = \sum_{i\geq 0} \ell^i \cdot \lambda_i$ and $\mu = \sum_{i\geq 0} \ell^i \cdot \mu_i$, with $\lambda_i,\mu_i\in X_1$ for all $i \geq 0$.
By iterating Steinberg's tensor product theorem, we obtain tensor product decompositions
\[ L(\lambda) \cong \bigotimes_{i\geq 0} L(\lambda_i)^{[i]} \qquad \text{and} \qquad L(\mu) \cong \bigotimes_{i\geq 0} L(\mu_i)^{[i]} . \]
We record a corollary of Lemma \ref{lem:Donkintensorproductindecomposable}, which allows us to describe $M(\lambda,\mu)$ as a tensor product of Frobenius twists of the different $M(\lambda_i,\mu_i)$ in many cases.

\begin{Corollary} \label{cor:highestweightdirectsummandFrobenius}
	Suppose that we are in the modular case, and that $M(\lambda_i,\mu_i)$ is indecomposable as a $\G_1$-module for all $i\geq 0$.
	Then
	\[ M(\lambda,\mu) \cong \bigotimes_{i \geq 0} M(\lambda_i,\mu_i)^{[i]} . \]
\end{Corollary}
\begin{proof}
	Since $M(\lambda_i,\mu_i)$ is a direct summand of $L(\lambda_i) \otimes L(\mu_i)$, the tensor product $M \coloneqq \bigotimes_i M(\lambda_i,\mu_i)^{[i]}$ is a direct summand of
	\[ L(\lambda) \otimes L(\mu) \cong \bigotimes_{i \geq 0} \big( L(\lambda_i) \otimes L(\mu_i) \big)^{[i]} . \]
	Furthermore, the $\lambda+\mu$-weight space of $M$ is non-zero because the $\lambda_i+\mu_i$ weight space of $M(\lambda_i,\mu_i)$ is non-zero for all $i$, so it remains to show that $M$ is indecomposable. This follows from Corollary \ref{cor:twistedtensorproductindecomposableiterated}, by our assumption on the $\G$-modules $M(\lambda_i,\mu_i)$.
\end{proof}

\section{Results in small rank} \label{sec:smallrank}

In this section, we compute examples of generic direct summands for $\G$ of type $\mathrm{A}_1$ and $\mathrm{A}_2$.
In each of the two cases, we consider first the generic direct summands $G(x,y)$ of tensor products $L(x\Cdot0) \otimes L(y\Cdot0)$ of simple $\G$-modules, and then turn our attention to the generic direct summands $G_\nabla(x,y)$ of tensor products $\nabla(x\Cdot0) \otimes \nabla(y\Cdot0)$ of dual Weyl modules, for $x,y \in W_\mathrm{ext}^+$.

\subsection{Type \texorpdfstring{$\mathrm{A}_1$}{A1}} \label{sec:A1}

Suppose (unless otherwise stated) that $\G$ is of type $\mathrm{A}_1$.
Then $\Phi^+ = \Pi = \{ \alpha_\mathrm{h} \}$ and we write $\alpha\coloneqq \alpha_\mathrm{h}$ for the unique positive root.
The weight lattice $X$ is a free $\Z$-module of rank $1$, spanned by the fundamental dominant weight $\varpi_\alpha$, and we can identify $X$ with $\Z$ via $\varpi_\alpha \mapsto 1$.
Under this identification, the unique positive root $\alpha$ is mapped to $2$, $\rho = \frac12 \cdot \alpha$ is mapped to $1$ and the scalar product $(-\,,-)$ on the euclidean space $X_\R = X \otimes_\Z \R \cong \R$ corresponds to the multiplication of real numbers.
Furthermore, the set $X^+$ of dominant weights is identified with the set $\Z_{\geq 0}$ of non-negative integers.
Accordingly, we denote the $\G$-modules $L(a\varpi_\alpha)$, $\nabla(a\varpi_\alpha)$, $\Delta(a\varpi_\alpha)$ and $T(a\varpi_\alpha)$ by $L(a)$, $\nabla(a)$, $\Delta(a)$ and $T(a)$, respectively, for $a\in\Z_{\geq 0}$.
The set of $\ell$-restricted weights is $X_1 = \{ 0 , \ldots , \ell-1 \}$.

\subsubsection*{Simple modules}

Our first aim is to determine the regular indecomposable direct summands of tensor products of simple $\G$-modules.
We will use the following \emph{Clebsch-Gordan formulas}:

\begin{Definition}
	For $a,b\in \Z_{\geq 0}$, we set
	\[ \mathrm{CG}(a,b) = \big\{ \abs{a-b}+2i \mathrel{\big|} i=0,\ldots,\min\{a,b\} \big\}\]
	and
	\[ \mathrm{CG}_\ell(a,b) \coloneqq \mathrm{CG}(a,b) \setminus \big\{ 2\ell-2-c \mathrel{\big|} c \in \mathrm{CG}(a,b) \text{ with } c \geq \ell \big\} . \]
\end{Definition}

\begin{Remark} \label{rem:exA1ClebschGordanclassical}
	The classical Clebsch-Gordan formula for $\mathfrak{g}=\mathfrak{sl}_2(\C)$ states that
	\[ L_\C(a) \otimes L_\C(b) \cong \bigoplus_{c \in \mathrm{CG}(a,b)} L_\C(c) \]
	for all $a,b\in \Z_{\geq 0}$.
	This can be proven by comparing characters.
\end{Remark}

The following analogue of the Clebsch-Gordan formula is proven in \cite[Lemma 1.3]{DotyHenke} in the modular case; the proof in the quantum case is completely analogous.
Note that the set $\mathrm{CG}_\ell(a,b)$ is denoted by $W(a,b)$ in \cite{DotyHenke}.

\begin{Lemma} \label{lem:exA1KrullSchidtrestricted}
	For $a,b\in X_1$, we have
	\[ L(a) \otimes L(b) \cong \bigoplus_{c \in \mathrm{CG}_\ell(a,b)} T(c) . \]
\end{Lemma}

Before we discuss generic direct sumands, let us say some words about the (extended) affine Weyl group of $\G$ and its (dot) action on $X\cong\Z$.
The finite Weyl group $W_\mathrm{fin}$ is cyclic of order $2$, generated by the reflection $s \coloneqq s_{\alpha}$, which acts on $X_\R \cong \R$ via $s(z)=-z$ for $z\in\R$.
The root lattice $\Z\Phi \subseteq X$ identifies with the set $2\Z$ of even integers, and it follows that $\Omega \cong X / \Z\Phi \cong \Z / 2\Z$ is also cyclic of order $2$.
The non trivial element of $\Omega$ is $w \coloneqq t_1s$ (where $t_1$ denotes the translation $z \mapsto z+1$ on $\R$, according to our identification of $X_\R$ with $\R$), and its (dot) action on $\R$ is given by $w\Cdot z = \ell-2-z$, for $z\in\R$.
For $x=t_a s^\varepsilon \in W_\mathrm{ext}$, with $a\in\Z$ and $\varepsilon\in\{0,1\}$, we have
\[ x\Cdot0 = t_a s^\varepsilon \Cdot 0 = \ell a - 2 \varepsilon = \ell \cdot (a - \varepsilon) + \varepsilon \cdot (\ell-2) = t_{a-\varepsilon} w^\varepsilon \Cdot 0 , \]
and it follows that, for all $x\in W_\mathrm{ext}$, we can choose an integer $b\in\Z$ and and element $\omega\in\Omega$ such that $x\Cdot0 = t_b \omega\Cdot0$. Furthermore, $b$ and $\omega$ are unique with this property if $\ell>2$, and we have $b\geq 0$ whenever $x\in W_\mathrm{ext}^+$. 

In the quantum case, the classical Clebsch-Gordan formula and the results from Section \ref{sec:SteinbergLusztigTPtheorem} allow us to determine all regular indecomposable direct summands of tensor products of simple $\G$-modules with highest weights in arbitrary $\ell$-alcoves.

\begin{Lemma} \label{lem:exA1regularpartquantum}
	Suppose that we are in the quantum case.
	For $x,y\in W_\mathrm{ext}^+$, let $a,b\in\Z_{\geq0}$ and $\omega,\omega^\prime\in\Omega$ such that $x\Cdot0=t_a \omega\Cdot0$ and $y\Cdot0=t_b \omega^\prime\Cdot0$. Then
	\[ \big( L(x\Cdot0) \otimes L(y\Cdot0) \big)_\mathrm{reg} \cong \bigoplus_{c\in \mathrm{CG}(a,b)} T^{\omega\omega^\prime} L(\ell c) \]
	and $G(x,y) \cong T^{\omega\omega^\prime} L\big( \ell \cdot (a+b) \big)$.
\end{Lemma}
\begin{proof}
	We have
	\[ \big( L(x\Cdot0) \otimes L(y\Cdot0) \big)_\mathrm{reg} \cong \big( T^\omega L(t_a\Cdot0) \otimes T^{\omega^\prime} L(t_b\Cdot0) \big)_\mathrm{reg} \cong T^{\omega\omega^\prime} \big( L(t_a\Cdot0) \otimes L(t_b\Cdot0) \big)_\mathrm{reg} \]
	by Lemma \ref{lem:translationtensorfundamentalgroup}, so it suffices to prove that
	\[ \big( L(t_a\Cdot0) \otimes L(t_b\Cdot0) \big)_\mathrm{reg} \cong \bigoplus_{c\in \mathrm{CG}(a,b)} L(\ell c) . \]
	Observe that $L(t_a\Cdot0) = L(\ell a) \cong L_\C(a)^{[1]}$ and $L(t_b\Cdot0) = L(\ell b) \cong L_\C(b)^{[1]}$, and therefore
	\[ L(t_a\Cdot0) \otimes L(t_b\Cdot0) \cong \big( L_\C(a) \otimes L_\C(b) \big)^{[1]} \cong \bigoplus_{c\in \mathrm{CG}(a,b)} L_\C(c)^{[1]} \]
	by Remark \ref{rem:exA1ClebschGordanclassical}.
	Now the first claim follows since $L_\C(c)^{[1]} \cong L(\ell c)$ is regular for all $c\in \mathrm{CG}(a,b)$, by Proposition \ref{prop:minimalcomplexsimplemodule}.
	Furthermore, we have $G(t_a,t_b) \cong L_\C(a+b)^{[1]}$ by Theorem \ref{thm:lusztigtensorproductgenericdirectsummand} because $G(e,e) \cong L(0)$ is the trivial $\G$-module, and so $G(x,y) \cong T^{\omega\omega^\prime} G(t_a,t_b) \cong T^{\omega\omega^\prime} L_\C(a+b)^{[1]}$ by Lemma \ref{lem:genericdirectsummandsimplemodulefundamentalgroup}.
\end{proof}

From now on until Lemma \ref{lem:exA1regularpartmodular} (included), suppose that we are in the modular case.
In order to state a modular analogue of Lemma \ref{lem:exA1regularpartquantum}, we need the following definition and theorem:

\begin{Definition}
	For a sequence $\mathbf{u} = (u_0,u_1,\ldots) \in \Z^\N$ with $0 \leq u_i \leq 2\ell-2$ for all $i$ and $u_i=0$ for all but finitely many $i$, let
	\[ J(\mathbf{u}) \coloneqq \bigotimes_{i\geq 0} T(u_i)^{[i]} . \]
\end{Definition}

The structure of the $\G$-modules $J(\mathbf{u})$ is studied in detail in \cite{DotyHenke}.
Theorem 2.1 in \cite{DotyHenke} gives the Krull-Schmidt decomposition of tensor products of simple $\G$-modules: 

\begin{Theorem}[Doty-Henke] \label{thm:exA1KrullSchmidt}
	Let $a,b\in\Z_{\geq0}$ and write $a=\sum_i a_i \ell^i$ and $b=\sum_i b_i \ell^i$ with $0 \leq a_i,b_i < \ell$ for all $i$.
	Then the Krull-Schmidt decomposition of $L(a) \otimes L(b)$ is given by
	\[ L(a) \otimes L(b) \cong \bigoplus_{\mathbf{u}} J(\mathbf{u}) , \]
	where the direct sum runs over all sequences $\mathbf{u}=(u_0,u_1,\ldots)$ with $u_i \in \mathrm{CG}_\ell(a_i,b_i)$ for all $i$.
\end{Theorem}

The modular analogue of Lemma \ref{lem:exA1regularpartquantum} is as follows:

\begin{Lemma} \label{lem:exA1regularpartmodular}
	Suppose that we are in the modular case.
	For $x,y\in W_\mathrm{ext}^+$, let $a,b\in\Z_{\geq0}$ and $\omega,\omega^\prime\in\Omega$ such that $x\Cdot0=t_a \omega\Cdot0$ and $y\Cdot0=t_b \omega^\prime\Cdot0$.
	Furthermore, write $a=\sum_i a_i \ell^i$ and $b=\sum_i b_i \ell^i$ with $0 \leq a_i < \ell$ and $0 \leq b_i < \ell$ for all $i$. Then
	\[ \big( L(x\Cdot0) \otimes L(y\Cdot0) \big)_\mathrm{reg} \cong \bigoplus_{\mathbf{u}} T^{\omega\omega^\prime} J(\mathbf{u})^{[1]} , \]
	where the direct sum runs over all sequences $\mathbf{u}=(u_0,u_1,\ldots)$ with $u_i \in \mathrm{CG}_\ell(a_i,b_i)$ for all $i$. Furthermore, with $\mathbf{a}=(a_0+b_0,a_1+b_1,\ldots)$, we have and $G(x,y) \cong T^{\omega\omega^\prime} J(\mathbf{a})^{[1]}$.
\end{Lemma}
\begin{proof}
	As in the proof of Lemma \ref{lem:exA1regularpartquantum}, we have
	\[ \big( L(x\Cdot0) \otimes L(y\Cdot0) \big)_\mathrm{reg} \cong \big( T^\omega L(t_a\Cdot0) \otimes T^{\omega^\prime} L(t_b\Cdot0) \big)_\mathrm{reg} \cong T^{\omega\omega^\prime} \big( L(t_a\Cdot0) \otimes L(t_b\Cdot0) \big)_\mathrm{reg} , \]
	and it suffices to prove that
	\[ \big( L(t_a\Cdot0) \otimes L(t_b\Cdot0) \big)_\mathrm{reg} \cong \bigoplus_{\mathbf{u}} J(\mathbf{u})^{[1]} . \]
	Observe that $L(t_a\Cdot0) = L(\ell a) \cong L(a)^{[1]}$ and $L(t_b\Cdot0) = L(\ell b) \cong L(b)^{[1]}$, and therefore
	\[ L(t_a\Cdot0) \otimes L(t_b\Cdot0) \cong \big( L(a) \otimes L(b) \big)^{[1]} \cong \bigoplus_{\mathbf{u}} J(\mathbf{u})^{[1]} \]
	by Theorem \ref{thm:exA1KrullSchmidt}.
	Now the first claim follows since $J(\mathbf{u})^{[1]}$ is regular for all $\mathbf{u}$ by Proposition \ref{prop:Frobeniustwistregular}.
	
	Furthermore, it is straightforward to see that $J(\mathbf{a}) \cong M(a,b)$ is the unique indecomposable direct summand of $L(a) \otimes L(b)$ with a non-zero $a+b$-weight space.
	As $G(e,e) \cong L(0)$ is indecomposable as a $\G_1$-module, we have $G(t_a,t_b) \cong M(a,b)^{[1]}$ by Theorem \ref{thm:steinbergtensorproductgenericdirectsummand}, and
	\[ G(x,y) = G(t_a\omega,t_b\omega^\prime) \cong T^{\omega\omega^\prime} G(t_a,t_b) \cong T^{\omega\omega^\prime} M(a,b)^{[1]} \cong T^{\omega\omega^\prime} J(\mathbf{a})^{[1]} \]
	by Lemma \ref{lem:genericdirectsummandsimplemodulefundamentalgroup}, as claimed.
\end{proof}

\subsection{Dual Weyl modules}

Now we turn our attention to generic direct summands of tensor products of dual Weyl modules (see Remark \ref{rem:genericdirectsummandinducedmodule}).
Our approach is based on results from \cite{Cavallin}, where the indecomposable direct summands of tensor products of dual Weyl modules were described as injective modules over certain Schur algebras.
We consider the modular case first.

\subsubsection*{The modular case}

We start by recalling some facts about polynomial representations and truncated categories.
For $d \geq 0$, consider the category $\mathrm{Pol}(2,d)$ of polynomial $\GL_2(\kk)$-modules of degree $d$, as in Section II.A.3 in \cite{Jantzen}.
The simple objects in $\mathrm{Pol}(2,d)$ are indexed by the set $\pi(2,d)$ of $2$-part partitions of $d$, and the restriction to $\G = \SL_2(\kk)$ of a simple $\GL_2(\kk)$-module corresponding to a partition $(a,b) \in \pi(2,d)$ is isomorphic to $L(a-b)$.
Let us write $\pi_d = \{ d , d-2 , d-4 , \ldots \} \cap \Z_{\geq 0}$, and note that $\pi(2,d)$ is in bijection with $\pi_d$ via $(a,b) \mapsto a-b$.

Now also consider the truncated category $\Rep(\G)_{\leq d}$, i.e.\ the Serre subcategory of $\Rep(\G)$ generated by the simple $\G$-modules $L(a)$ with $a \in \pi_d$.
The restriction functor from $\mathrm{Pol}(2,d)$ to $\Rep(G)_{\leq d}$ is an equivalence of categories because the center of $\GL_2(\kk)$ acts on all objects of $\mathrm{Pol}(2,d)$ via the same central character.
(Alternatively, it suffices to note that the restriction functor is an exact functor between highest weight categories which induces a bijection between the standard objects and the costandard objects in the two categories, as explained in \cite[Section 1.2]{ArkhipovBezrukavnikovGinzburg}.)
By Section II.A.6 in \cite{Jantzen}, every simple $\G$-module $L(a)$ with $a \in \pi_d$ has an injective hull $I_d(a)$ in $\Rep(\G)_{\leq d}$, and $I_d(a)$ admits a \emph{good filtration} (see Section II.4.16 in \cite{Jantzen}).
The multiplicity $[ I_d(a) : \nabla(b)]_\nabla$ of $\nabla(b)$ in a good filtration of $I_d(a)$ coincides with the composition mutliplicity $[\nabla(b):L(a)]$ for $b \in \pi_d$.

Let us write $E = \kk^2$ for the vector representation of $\GL_2(\kk)$.
Then the symmetric power $S^d E$ is a polynomial representation of degree $d$ for all $d \geq 0$, and its restriction to $\G = \SL_2(\kk)$ is the dual Weyl module $\nabla(d)$.
More generally, the tensor product $S^a E \otimes S^b E$ is an injective object of $\mathrm{Pol}(2,a+b)$ for all $a,b \geq 0$ by Section 2.1(8) in \cite{DonkinqSchuralgebra}, and it follows that $\nabla(a) \otimes \nabla(b)$ is injective in $\Rep(\G)_{\leq d}$.
In particular $\nabla(a) \otimes \nabla(b)$ decomposes as a direct sum of $\G$-modules of the form $I_d(c)$ for certain $c \in \pi_{a+b}$.

The following preliminary lemma will be useful for our description of generic direct summands of tensor products of dual Weyl modules.

\begin{Lemma} \label{lem:exA1dualWeylhighestweightsummand}
	Let $a,b \geq 0$ and write $a = \sum_i a_i \ell^i$ and $b = \sum_i b_i \ell^i$, with $0 \leq a_i < \ell$ and $0 \leq b_i < b$ for all $i \geq 0$.
	Further set $c_i = \min \{ a_i+b_i , 2\ell-2-a_i-b_i \}$ for all $i \geq 0$, and define
	\[ c = c(a,b) \coloneqq \sum_i c_i \ell^i . \]
	Then $I_{a+b}(c)$ is the unique indecomposable direct summand of the tensor product $\nabla(a) \otimes \nabla(b)$ with a non-zero $(a+b)$-weight space.
\end{Lemma}
\begin{proof}
	Note that $L(c_i)$ embeds into $L(a_i) \otimes L(b_i)$ for all $i \geq 0$ by Lemmas 1.1 and 1.3 in \cite{DotyHenke}.
	Using Steinberg's tensor product theorem, it follows that there is an embedding
	\[ L(c) \cong \bigotimes_i L(c_i)^{[i]} \longrightarrow \bigotimes_i \big(L(a_i) \otimes L(b_i)  \big)^{[i]} \cong L(a) \otimes L(b) \longrightarrow \nabla(a) \otimes \nabla(b) , \]
	and since the tensor product $\nabla(a) \otimes \nabla(b)$ is injective in $\Rep(\G)_{\leq a+b}$, it follows that $I_{a+b}(c)$ is a direct summand of $\nabla(a) \otimes \nabla(b)$.
	Furthermore, we have
	\[ [ I_{a+b}(c) : \nabla(a+b) ]_\nabla = [ \nabla(a+b) : L(c) ] = 1 \]
	by Theorem 2.1 in \cite{HenkeCartanMatrix}, whence $I_{a+b}(c)$ has a non-zero $(a+b)$-weight space.
\end{proof}

Recall from the discussion above Lemma \ref{lem:exA1regularpartquantum} that for $x \in W_\mathrm{ext}^+$, we can choose $a \in \Z_{\geq 0}$ and $\omega \in \Omega$ such that $x\Cdot0 = t_a \omega \Cdot 0$.
The generic direct summands of tensor products of dual Weyl modules can be described as follows:

\begin{Lemma} \label{lem:exA1genericdualWeylmodular}
	For $x,y \in W_\mathrm{ext}^+$, let $a,b\in\Z_{\geq0}$ and $\omega,\omega^\prime\in\Omega$ such that $x\Cdot0=t_a \omega\Cdot0$ and $y\Cdot0=t_b \omega^\prime\Cdot0$.
	Define $c = c(a,b)$ as in Lemma \ref{lem:exA1dualWeylhighestweightsummand}.
	Then $G_\nabla(x,y) = T^{\omega\omega^\prime} I_{\ell a + \ell b}(\ell \cdot c)$.
\end{Lemma}
\begin{proof}
	By Lemma \ref{lem:genericdirectsummandWeylmodulefundamentalgroup}, we have $G_\nabla(x,y) \cong T^{\omega\omega^\prime} G_\nabla(t_a,t_b)$, so it suffices to prove that
	\[ G_\nabla(t_a,t_b) \cong I_{\ell a + \ell b}(\ell \cdot c) . \]
	Let us write $I = I_{\ell a + \ell b}(\ell \cdot c)$, and note that $I$ is a direct summand of $\nabla(t_a\Cdot0) \otimes \nabla(t_b\Cdot0) = \nabla(\ell a) \otimes \nabla(\ell b)$ by Lemma \ref{lem:exA1dualWeylhighestweightsummand}.
	As $G_\nabla(t_a,t_b)$ has Weyl filtration dimension $\ell(t_a)+\ell(t_b) = a+b$ and belongs to the linkage class of $\omega_{t_a t_b}\Cdot0 = \omega_{t_{a+b}} \Cdot 0$ (see Remark \ref{rem:genericdirectsummandinducedmodule}), the claim follows if we prove that every indecomposable direct summand $M \ncong I$ of $\nabla(\ell a) \otimes \nabla(\ell b)$ that belongs to the linkage class of $\omega_{t_{a+b}} \Cdot 0$ satisfies $\wfd(M)<a+b$.
	Indeed, by weight considerations, every composition factor of $M$ has highest weight in an alcove of the form $t_d \Cdot C_\mathrm{fund}$ for some $d\in\Z_{\geq0}$ with $d<a+b$.
	Thus, all composition factors of $M$ have Weyl filtration dimension at most $a+b-1$ by Remark 4.2 in \cite{GruberLinkageTranslation}, and using Lemma 1.14(5) in \cite{GruberMinimalTilting}, it follows that $\wfd(M) < a+b$, as required.
\end{proof}

\begin{Remark}
	Let $x,y \in W_\mathrm{ext}^+$ and choose $a,b \in \Z_{\geq 0}$ and $\omega,\omega^\prime \in \Omega$ such that $x\Cdot0=t_a \omega\Cdot0$ and $y\Cdot0=t_b \omega^\prime\Cdot0$..
	Further write $a = \sum_i a_i \ell^i$ and $b = \sum_i b_i \ell^i$, with $0 \leq a_i < \ell$ and $0 \leq b_i < b$ for all $i \geq 0$, and suppose that $a_i +b_i \leq \ell-1$ for all $i \geq 0$.
	Then $c = c(a,b) = a+b$, and it follows that
	\[ G_\nabla(x,y) \cong T^{\omega\omega^\prime} I_{\ell a + \ell b}( \ell a + \ell b ) \cong T^{\omega\omega^\prime} \nabla( \ell a + \ell b ) \cong \nabla(t_{a+b} \omega\omega^\prime \Cdot0) , \]
	by Lemma \ref{lem:exA1genericdualWeylmodular} and \cite[Section II.A.6]{Jantzen}.
\end{Remark}

\subsubsection*{The quantum case}

In order to prove a quantum analogue of Lemma \ref{lem:exA1genericdualWeylmodular}, we will use the $\zeta$-Schur algebras $S_\zeta(2,d)$ from \cite{DonkinqSchuralgebra}, where $\zeta \in \C^\times$ is the root of unity in $\G = U_\zeta(\mathfrak{g})$ and $\mathfrak{g} = \mathfrak{sl}_2(\C)$.
We write $\mathrm{Pol}_\zeta(2,d)$ for the category of finite-dimensional left $S_\zeta(2,d)$-modules.
This is a highest weight category with simple objects indexed by the set $\pi(2,d)$ of $2$-part partitions of $d$.
The $\zeta$-Schur algebra $S_\zeta(2,1)$ has a $2$-dimensional vector representation $E$, and for $d \geq 1$, the tensor power $E^{\otimes d}$ and the quantum symmetric power $S^d E$ are representations of $S_\zeta(2,d)$.
(See the introduction to Section 2.1 in \cite{DonkinqSchuralgebra}.)
Furthermore, $S^d E$ is the costandard object of highest weight $(d,0) \in \pi(2,d)$ in $\mathrm{Pol}_\zeta(2,d)$; see \cite[Section 2.1.(15)]{DonkinqSchuralgebra}.

The $\zeta$-Schur $S_\zeta(2,d)$ is a quotient of the quantized enveloping algebra $U_\zeta( \mathfrak{gl}_2(\C) )$ by Theorem 3.4 in \cite{DuQuantizedWeylReciprocity}, and so there is a canonical restriction functor from $\mathrm{Pol}_\zeta(2,d)$ to $\Rep(\G)$.
As in the modular case, it sends the simple $S_\zeta(2,d)$-module of highest weight $(a,b) \in \pi(2,d)$ to the simple $\G$-module $L(a-b)$ and induces an equivalence between $\mathrm{Pol}_\zeta(2,d)$ and the Serre subcategory $\Rep(\G)_{\leq d}$ of $\Rep(\G)$ generated by the simple $\G$-modules $L(c)$ with $c \in \pi_d = \{ d , d-2 , d-4 , \ldots \} \cap \Z_{\geq 0}$.
(As before, this follows from the fact that the restriction functor is an exact functor between highest weight categories which induces bijections between the standard objects and between the costandard objects in $\mathrm{Pol}_\zeta(2,d)$ and $\Rep(\G)_{\leq d}$, cf.\ Section 1.2 in \cite{ArkhipovBezrukavnikovGinzburg}.)

By Section 2.1.(8) in \cite{DonkinqSchuralgebra}, the tensor product of symmetric powers $S^a E \otimes S^b E$ is injective in $\mathrm{Pol}_\zeta(2,d)$ for all $a,b \geq 0$, and it follows that $\nabla(a) \otimes \nabla(b)$ is injective in $\Rep(\G)_{\leq a+b}$.
This fact allows us to prove a quantum analogue of Lemma \ref{lem:exA1genericdualWeylmodular}.

\begin{Lemma}
	For $x,y \in W_\mathrm{ext}^+$, let $a,b\in\Z_{\geq0}$ and $\omega,\omega^\prime\in\Omega$ such that $x\Cdot0=t_a \omega\Cdot0$ and $y\Cdot0=t_b \omega^\prime\Cdot0$.
	Then $G_\nabla(x,y) = T^{\omega\omega^\prime} \nabla(\ell a + \ell b)$.
\end{Lemma}
\begin{proof}
	As in the modular case (see Lemma \ref{lem:exA1genericdualWeylmodular}), it suffices to prove that $G_\nabla(t_a,t_b) \cong \nabla( \ell a + \ell b )$ by Lemma \ref{lem:genericdirectsummandWeylmodulefundamentalgroup}.
	Observe that there is an embedding
	\[ L( \ell a + \ell b ) \cong L_\C(a+b)^{[1]} \longrightarrow L_\C(a)^{[1]} \otimes L_\C(b)^{[1]} \cong L(\ell a) \otimes L(\ell b) \longrightarrow \nabla(\ell a) \otimes \nabla(\ell b) . \]
	As $\nabla(\ell a) \otimes \nabla(\ell b)$ is injective in the truncated category $\Rep(\G)_{\leq \ell a + \ell b}$, this implies that the injective hull $I_{\ell a+ \ell b}(\ell a + \ell b)$ of $L(\ell a + \ell b)$ in $\Rep(\G)_{\leq \ell a + \ell b}$ is a direct summand of $\nabla(\ell a) \otimes \nabla(\ell b)$.
	Now $\ell a + \ell b$ is maximal in $\pi_{\ell a+ \ell b}$, so $I_{\ell a+ \ell b}(\ell a + \ell b) \cong \nabla(\ell a + \ell b)$ by Section II.A.6 in \cite{Jantzen}, and as $\nabla(\ell a + \ell b)$ is regular by Proposition \ref{prop:minimalcomplexWeylmodule} (and taking duals), it follows that $\nabla(\ell a + \ell b) \cong G_\nabla(t_a,t_b)$, as required.
\end{proof}

\subsection{Type \texorpdfstring{$\mathrm{A}_2$}{A2}} \label{sec:A2}

Suppose that $\G$ is of type $\mathrm{A}_2$ and $\ell \geq 3$.
Then the simple roots and positive roots are given by $\Pi = \{ \alpha_1 , \alpha_2 \}$ and $\Phi^+ = \{ \alpha_1,\alpha_2,\alpha_\mathrm{h} \}$, with $\alpha_\mathrm{h}=\alpha_1+\alpha_2$, and the weight lattice $X \cong \Z^2$ is spanned by the fundamental dominant weights $\varpi_1 = \varpi_{\alpha_1}$ and $\varpi_2 = \varpi_{\alpha_2}$.
The affine Weyl group $W_\mathrm{aff}$ is generated by the simple reflections $S = \{ s , t , u \}$, where $s=s_{\alpha_1}$, $t=s_{\alpha_2}$ and $u=s_{\alpha_\mathrm{h},1}$.
Recall that $W_\mathrm{aff}$ is in bijection with the set of alcoves in $X_\R$ via $x \mapsto x \Cdot C_\mathrm{fund}$.
In Figure \ref{fig:alcovesA2}, we display some alcoves for $\G$, and we label some of them by the corresponding elements of $W_\mathrm{aff}$.
The only $\ell$-alcoves containing $\ell$-restricted weights are $C_\mathrm{fund}$ and $u\Cdot C_\mathrm{fund}$.

Our discussion of regular indecomposable direct summands (and generic direct summands) of tensor products of $\G$-modules strongly relies on the two articles \cite{BowmanDotyMartin} (by C.~Bowman, S.~Doty and S.~Martin) and \cite{ChenWang} (by X.~Chan and J.~Wang), whose contents we will briefly discuss here.
The main result of \cite{BowmanDotyMartin} is a description of the set of indecomposable $\G$-modules that arise as direct summands of tensor products of simple $\G$-modules with $\ell$-restricted highest weights.
Strictly speaking, the article only covers the modular case, but none of its methods are specific to that case, and the results that we will use hold in the quantum case as well.
The analogous problem (of finding the indecomposable direct summands of tensor products) for dual Weyl modules with $\ell$-restricted highest weights was considered in \cite{ChenWang}.
Again, the authors discuss only the modular case, but the results are valid in the quantum case as well (as is pointed out at the end of the introduction of that article).
We can combine the results of \cite{BowmanDotyMartin} with the techniques developed in Section \ref{sec:SteinbergLusztigTPtheorem} to give a complete description of the generic direct summands of tensor products of simple $\G$-modules with arbitrary $\ell$-regular (not necessarily $\ell$-restricted) highest weights.
For dual Weyl modules, there is (at present) no method for reducing the study of generic direct summands to the case of $\ell$-restricted highest weights, so we do not go any further than to point out which of the indecomposable $\G$-modules from \cite{ChenWang} are the generic direct summands.

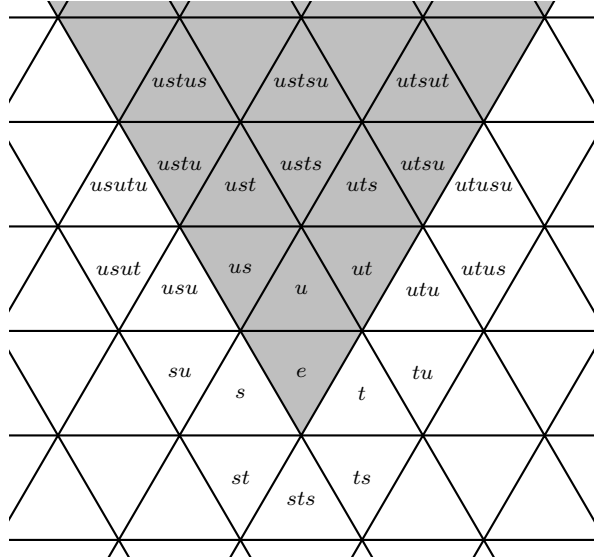
\begin{figure}
		\centering
		\begin{tikzpicture}[scale=1.6]
		\clip (-2.4,-1) rectangle (2.4,3.6);
		
		\pgftransformcm{cos(60)}{sin(60)}{cos(120)}{sin(120)}{\pgfpoint{0cm}{0cm}}
		
		\fill[lightgray] (0,0) rectangle (10,10);
		
		\draw[thick] (-5,-5) grid (10,10);
		\foreach \x in {-5,-4,...,9}{
			\foreach \y in {-4,-3,...,10}{
				\draw[thick]  (\x,\y) -- (\x+1,\y-1);
			}
		}
		\node at (.3,.3) {\scriptsize $e$};
		\node at (.7,.7) {\scriptsize $u$};
		\node at (-.3,.7) {\scriptsize $s$};
		\node at (.7,-.3) {\scriptsize $t$};
		\node at (.3,1.3) {\scriptsize $us$};
		\node at (1.3,.3) {\scriptsize $ut$};
		\node at (.7,1.7) {\scriptsize $ust$};
		\node at (1.7,.7) {\scriptsize $uts$};
		\node at (1.3,1.3) {\scriptsize $usts$};
		\node at (0.3,2.3) {\scriptsize $ustu$};
		\node at (2.3,0.3) {\scriptsize $utsu$};
		\node at (0.7,2.7) {\scriptsize $ustus$};
		\node at (2.7,0.7) {\scriptsize $utsut$};
		\node at (1.7,1.7) {\scriptsize $ustsu$};
		\node at (-.7,1.3) {\scriptsize $su$};
		\node at (-.3,1.7) {\scriptsize $usu$};
		\node at (-.7,2.3) {\scriptsize $usut$};
		\node at (-.3,2.7) {\scriptsize $usutu$};
		\node at (1.3,-.7) {\scriptsize $tu$};
		\node at (1.7,-.3) {\scriptsize $utu$};
		\node at (-.7,.3) {\scriptsize $st$};
		\node at (.3,-.7) {\scriptsize $ts$};
		\node at (-.3,-.3) {\scriptsize $sts$};
		\node at (2.3,-.7) {\scriptsize $utus$};
		\node at (2.7,-.3) {\scriptsize $utusu$};
		\end{tikzpicture}
		\caption[Alcoves for $\G$ of type $\mathrm{A}_2$]{Alcoves for $\G$ of type $\mathrm{A}_2$. For some elements $x\in W_\mathrm{aff}$, we have labeled the alcove $x\Cdot C_\mathrm{fund}$ by $x$. The gray-colored region corresponds to $W_\mathrm{aff}^+$.}
		\label{fig:alcovesA2}
\end{figure}

\subsection*{Simple modules}

Let us start by recalling the results of C.~Bowman, S.~Doty and S.~Martin in more detail.
According to Section 8.5 in \cite{BowmanDotyMartin}, the tilting module $T(ust\Cdot\nu)$, for $\nu \in C_\mathrm{fund} \cap X$, has a unique submodule $M(\nu)$ (denoted there by $M(u\Cdot\nu)$) with $M(\nu) \cong M(\nu)^\tau$ and whose submodule structure is given by the following ``Alperin diagram" (where we replace the simple $\G$-module $L(x\Cdot\nu)$ by the label $x \in W_\mathrm{aff}^+$):
\phantomsection \label{def:Mnu}
\[ M(\nu) = 
	\begin{tikzpicture}[scale=.6,baseline={([yshift=-.5ex]current bounding box.center)}]
	\node (C1) at (0,0) {\small $u$};
	\node (C2) at (-1.5,1.5) {\small $us$};
	\node (C3) at (0,1.5) {\small $e$};
	\node (C4) at (1.5,1.5) {\small $ut$};
	\node (C5) at (0,3) {\small $u$};
	
	\draw (C1) -- (C2);
	\draw (C1) -- (C3);
	\draw (C1) -- (C4);
	\draw (C2) -- (C5);
	\draw (C3) -- (C5);
	\draw (C4) -- (C5);
	\end{tikzpicture} \]
Observe that the definition of $M(\nu)$ implies that $T_\nu^\lambda M(\nu) \cong M(\lambda)$ for all $\lambda \in C_\mathrm{fund} \cap X$.
With this notation in place, a simplified version of the main result of \cite{BowmanDotyMartin} can be stated as follows:

\begin{Theorem} \label{thm:BowmanDotyMartin}
	Let $\lambda,\mu \in X_1$ and let $M$ be an indecomposable direct summand of  $L(\lambda) \otimes L(\mu)$. Then either $M$ is a tilting module or $M=L(u\Cdot\nu)$ or $M \cong M(\nu)$, for some $\nu \in C_\mathrm{fund} \cap X$. The third case $M \cong M(\nu)$ can occur only if $\lambda , \mu \in u \Cdot C_\mathrm{fund}$.
\end{Theorem}

\begin{Remark} \label{rem:indecomposableasG1moduletypeA2}
	It is pointed out at the end of Section 3 (below Theorem B) in \cite{BowmanDotyMartin} that an indecomposable direct summand $M$ of $L(\lambda) \otimes L(\mu)$ as in the preceding theorem has simple socle with $\ell$-restricted highest weight as a $\G_1$-module, unless possibly when
	\[ M \cong T(ustus\Cdot\nu) \qquad \text{or} \qquad M \cong T(utsut\Cdot\nu) \]
	for some $\nu \in \overline{C}_\mathrm{fund} \cap X$.
	In particular, $M$ is indecomposable as a $\G_1$-module (with the same exceptions). 
\end{Remark}

\begin{Lemma} \label{lem:regularpartA2}
	We have $G(u,u) \cong M(0)$ and
	\[ \big( L(u \Cdot 0) \otimes L(u\Cdot 0) \big)_\mathrm{reg} \cong M(0) \oplus L(0) \]
\end{Lemma}
\begin{proof}
	Recall that $G(u,u)$ belongs to the linkage class of $0$ and that $\gfd\big( G(u,u) \big) = \ell(u) + \ell(u) = 2$; see Theorem \ref{thm:genericdirectsummandsimplemodule}.
	By Theorem \ref{thm:BowmanDotyMartin}, all indecomposable direct summands of $L(u \Cdot 0) \otimes L(u\Cdot 0)$ that are not of the form $M(\nu)$, for some $\nu \in C_\mathrm{fund} \cap X$, have good filtration dimension either zero (because they are tilting modules) or one (because $\gfd( L(u\Cdot\nu) ) = \ell(u) = 1$), and it follows that
	\[ G(u,u) \cong M(0) . \]
	In particular, $M(\nu) \cong T_0^\nu M(0)$ is regular for all $\nu \in C_\mathrm{fund} \cap X$.
	
	By Section 8.6 in \cite{BowmanDotyMartin}, there exist $\lambda,\mu,\nu \in C_\mathrm{fund} \cap X$ such that
	\[ \pr_\nu\big( L(u\Cdot\lambda) \otimes L(u\Cdot\mu) \big) \cong M(\nu) \oplus L(\nu) . \]
	Since $M(\nu)$ and $L(\nu)$ are both regular, Theorem \ref{thm:translationtensor} yields
	\[ T_0^\nu \big( M(0) \oplus L(0) \big) \cong M(\nu) \oplus L(\nu) \cong \pr_\nu\big( L(u\Cdot\lambda) \otimes L(u\Cdot\mu) \big)_\mathrm{reg} \cong T_0^\nu \big( L(u\Cdot0) \otimes L(u\Cdot0) \big)_\mathrm{reg}^{\oplus c_{\lambda,\mu}^\nu} , \]
	and we conclude that $c_{\lambda,\mu}^\nu=1$ and $\big( L(u\Cdot0) \otimes L(u\Cdot0) \big)_\mathrm{reg} \cong M(0) \oplus L(0)$.
\end{proof}

Also note that we have
\[ G(e,e) \cong L(0) \cong \big( L(0) \otimes L(0) \big)_\mathrm{reg} \qquad \text{and} \qquad G(u,e) \cong L(u\Cdot0) \cong \big( L(u\Cdot0) \otimes L(0) \big)_\mathrm{reg} . \]
As $C_\mathrm{fund}$ and $u\Cdot C_\mathrm{fund}$ are the only $\ell$-alcoves containing $\ell$-restricted weights, this gives a complete description of the regular parts (and the generic direct summands) of tensor products of simple $\G$-modules with $\ell$-regular $\ell$-restricted highest weights.
Furthermore, all regular indecomposable direct summands of such tensor products are strongly regular (because simple $\G$-modules with $\ell$-regular highest weights and generic direct summands of tensor products of simple $\G$-modules are strongly regular, by \cite[Remark 4.2]{GruberLinkageTranslation} and Remark \ref{rem:genericdirectsummandstronglyregular}) and indecomposable
(with simple socle)
as $\G_1$-modules by Remark \ref{rem:indecomposableasG1moduletypeA2}.
(These are important observations in view of Remarks \ref{rem:regularpartFrobeniusquantum} and \ref{rem:regularpartFrobeniusmodular}.)

Now as in Section \ref{sec:SteinbergLusztigTPtheorem}, let us fix $x,y\in W_\mathrm{ext}^+$ and write
\[ x\Cdot0 = x_0\Cdot0 + \ell\lambda \qquad \text{and} \qquad y\Cdot0 = y_0\Cdot0 + \ell\mu \]
with $\lambda,\mu\in X^+$ and $x_0,y_0\in W_\mathrm{ext}^+$ such that $x_0\Cdot0,y_0\Cdot0\in X_1$.
As $C_\mathrm{fund}$ and $u\Cdot C_\mathrm{fund}$ are the only $\ell$-alcoves containing $\ell$-restricted weights, we have $x_0 , y_0 \in \Omega \cup u\Omega$, and we write $x_0 = u^\varepsilon \omega$ and $y_0 = u^{\varepsilon^\prime} \omega^\prime$ with $\varepsilon,\varepsilon^\prime \in \{0,1\}$ and $\omega,\omega^\prime \in \Omega$.
Note that by Lemma \ref{lem:translationtensorfundamentalgroup}, we have
\[ \big( L(x_0\Cdot0) \otimes L(y_0\Cdot0) \big)_\mathrm{reg} \cong \big( T^\omega L(u^\varepsilon\Cdot0) \otimes T^{\omega^\prime} L(u^{\varepsilon^\prime}\Cdot0) \big)_\mathrm{reg} \cong T^{\omega\omega^\prime} \big( L(u^\varepsilon\Cdot0) \otimes L(u^{\varepsilon^\prime}\Cdot0) \big)_\mathrm{reg} . \]
In the quantum case, we get the following complete desccription of regular parts and generic direct summands of tensor products of simple $\G$-modules from Theorem \ref{thm:lusztigtensorproductgenericdirectsummand} and Remark \ref{rem:regularpartFrobeniusquantum}.

\begin{Theorem} \label{thm:exA2quantum}
	Suppose that we are in the quantum case and write
	\[ L_\C(\lambda) \otimes L_\C(\mu) \cong \bigoplus_{\nu \in X^+} L_\C(\nu)^{\oplus d_{\lambda,\mu}^\nu} . \]
	\begin{enumerate}
	\item If $\varepsilon = \varepsilon^\prime = 0$ then
	\begin{align*}
	\big( L(x\Cdot0) \otimes L(y\Cdot0) \big)_\mathrm{reg} & \cong \bigoplus_{\nu \in X^+} \big( L(\omega\omega^\prime\Cdot0) \otimes L_\C(\nu)^{[1]} \big)^{\oplus d_{\lambda,\mu}^\nu} \\
	& \cong \bigoplus_{\nu \in X^+} L(\omega\omega^\prime\Cdot0+\ell\nu)^{\oplus d_{\lambda,\mu}^\nu}
	\end{align*}
	is a Krull-Schmidt decomposition and $G(x,y) \cong L(\omega\omega^\prime\Cdot0) \otimes L_\C(\lambda+\mu)^{[1]} \cong L(\omega\omega^\prime\Cdot0+\ell\lambda+\ell\mu)$.
	\item If $\varepsilon + \varepsilon^\prime = 1$ then
	\begin{align*}
	\big( L(x\Cdot0) \otimes L(y\Cdot0) \big)_\mathrm{reg} & \cong \bigoplus_{\nu \in X^+} \big( L(u\omega\omega^\prime\Cdot0) \otimes L_\C(\nu)^{[1]} \big)^{\oplus d_{\lambda,\mu}^\nu} \\
	& \cong \bigoplus_{\nu \in X^+} L(u\omega\omega^\prime\Cdot0+\ell\nu)^{\oplus d_{\lambda,\mu}^\nu}
	\end{align*}
	is a Krull-Schmidt decomposition and $G(x,y) \cong L(u\omega\omega^\prime\Cdot0) \otimes L_\C(\lambda+\mu)^{[1]} \cong L(u\omega\omega^\prime\Cdot0+\ell\lambda+\ell\mu)$.
	\item If $\varepsilon = \varepsilon^\prime = 1$ then
	\[ \big( L(x\Cdot0) \otimes L(y\Cdot0) \big)_\mathrm{reg} \cong \bigoplus_{\nu \in X^+} \big( M(\omega\omega^\prime\Cdot0) \otimes L_\C(\nu)^{[1]} \big)^{\oplus d_{\lambda,\mu}^\nu} \oplus \bigoplus_{\nu \in X^+} L\big( \omega\omega^\prime\Cdot0 + \ell\nu \big)^{\oplus d_{\lambda,\mu}^\nu} \]
	is a Krull-Schmidt decomposition and $G(x,y) \cong M(\omega\omega^\prime\Cdot0) \otimes L_\C(\lambda+\mu)^{[1]}$.
	\end{enumerate}
\end{Theorem}

Recall that in the modular case, we write $M(\lambda,\mu)$ for the unique indecomposable direct summand of $L(\lambda) \otimes L(\mu)$ with a non-zero $(\lambda+\mu)$-weight space.
The modular analogue of the preceding theorem follows from Theorem \ref{thm:steinbergtensorproductgenericdirectsummand} and Remark \ref{rem:regularpartFrobeniusmodular}.

\begin{Theorem} \label{thm:exA2modular}
	Suppose that we are in the modular case and fix a Krull-Schmidt decomposition
	\[ L(\lambda) \otimes L(\mu) \cong M_1 \oplus \cdots \oplus M_r . \]
	\begin{enumerate}
	\item If $\varepsilon = \varepsilon^\prime = 0$ then
	\[ \big( L(x\Cdot0) \otimes L(y\Cdot0) \big)_\mathrm{reg} \cong \bigoplus_{i=1}^r L(\omega\omega^\prime\Cdot0) \otimes M_i^{[1]} \]
	is a Krull-Schmidt decomposition and $G(x,y) \cong L(\omega\omega^\prime\Cdot0) \otimes M(\lambda,\mu)^{[1]}$.
	\item If $\varepsilon + \varepsilon^\prime = 1$ then
	\[ \big( L(x\Cdot0) \otimes L(y\Cdot0) \big)_\mathrm{reg} \cong \bigoplus_{i=1}^r L(u\omega\omega^\prime\Cdot0) \otimes M_i^{[1]} \]
	is a Krull-Schmidt decomposition and $G(x,y) \cong L(u\omega\omega^\prime\Cdot0) \otimes M(\lambda,\mu)^{[1]}$.
	\item If $\varepsilon = \varepsilon^\prime = 1$ then
	\[ \big( L(x\Cdot0) \otimes L(y\Cdot0) \big)_\mathrm{reg} \cong \bigoplus_{i=1}^r M(\omega\omega^\prime\Cdot0) \otimes M_i^{[1]} \oplus \bigoplus_{i=1}^r L(\omega\omega^\prime\Cdot0) \otimes M_i^{[1]} \]
	is a Krull-Schmidt decomposition and $G(x,y) \cong M(\omega\omega^\prime\Cdot0) \otimes M(\lambda,\mu)^{[1]}$.
	\end{enumerate}
\end{Theorem}

For the rest of this section, suppose that we are in the modular case.
In order to complete the description of generic direct summands of tensor products of simple $\G$-modules, it remains to determine the $\G$-modules $M(\lambda,\mu)$, for $\lambda,\mu \in X^+$.
We start by determining the $\G$-modules $M(\lambda^\prime,\mu^\prime)$ with $\lambda^\prime,\mu^\prime\in X_1$ explicitly.

\begin{Lemma} \label{lem:highestweightdirectsummandtypeA2}
	Let $\lambda^\prime,\mu^\prime \in X_1$. We have
	\[ M(\lambda^\prime,\mu^\prime) \cong L(\lambda^\prime+\mu^\prime) \]
	whenever $\lambda^\prime + \mu^\prime \in u \Cdot C_\mathrm{fund}$ and either $\lambda^\prime\in u \Cdot C_\mathrm{fund}$ or $\mu^\prime \in u \Cdot C_\mathrm{fund}$, and
	\[ M(\lambda^\prime,\mu^\prime) \cong T(\lambda^\prime+\mu^\prime) \]
	in all other cases.
\end{Lemma}
\begin{proof}
	First suppose that neither of $\lambda^\prime$ and $\mu^\prime$ belongs to $u\Cdot C_\mathrm{fund}$. As $\lambda^\prime,\mu^\prime \in X_1 \subseteq \overline{C}_\mathrm{fund} \cup u \Cdot \overline{C}_\mathrm{fund}$, this implies that $\lambda^\prime,\mu^\prime\in C_\mathrm{fund}$ or that at least one of $\lambda^\prime$ and $\mu^\prime$ is $\ell$-singular. If $\lambda^\prime,\mu^\prime\in C_\mathrm{fund}$ then
	\[ L(\lambda^\prime) \otimes L(\mu^\prime) \cong T(\lambda^\prime) \otimes T(\mu^\prime) \]
	is a tilting module and it follows that $M(\lambda^\prime,\mu^\prime) \cong T(\lambda^\prime+\mu^\prime)$.
	If either of the weights $\lambda^\prime$ or $\mu^\prime$ is $\ell$-singular then one of the simple modules $L(\lambda^\prime)$ and $L(\mu^\prime)$ is singular by \cite[Lemma 3.3]{GruberLinkageTranslation}, and $M(\lambda^\prime,\mu^\prime)$ is singular because singular modules form a thick tensor ideal.
	On the other hand, $L(u\Cdot\nu)$ is regular for all $\nu \in C_\mathrm{fund} \cap X$ (again by \cite[Lemma 3.3]{GruberLinkageTranslation}), and Theorem \ref{thm:BowmanDotyMartin} implies that $M(\lambda^\prime,\mu^\prime)$ is a tilting module, so $M(\lambda^\prime,\mu^\prime) \cong T(\lambda^\prime+\mu^\prime)$.
	
	By symmetry, we may now suppose that $\lambda^\prime \in u \Cdot C_\mathrm{fund}$ and that $\mu^\prime$ is $\ell$-regular.
	If $\lambda^\prime+\mu^\prime\in u \Cdot C_\mathrm{fund}$ then $\mu^\prime$ is the unique dominant weight in the $W_\mathrm{fin}$-orbit of $u \Cdot(\lambda^\prime+\mu^\prime) - u \Cdot \lambda^\prime = s_{\alpha_\mathrm{h}}(\mu^\prime)$, and it follows that
	\[ L(\lambda^\prime+\mu^\prime) \cong T_{u\Cdot\lambda^\prime}^{u\Cdot(\lambda^\prime+\mu^\prime)} L(\lambda^\prime) = \pr_{u\Cdot(\lambda^\prime+\mu^\prime)}\big( L(\lambda^\prime) \otimes L(\mu^\prime) \big) , \]
	whence $M(\lambda^\prime,\mu^\prime) \cong L(\lambda^\prime+\mu^\prime)$.
	If $\mu^\prime \in C_\mathrm{fund}$ and $\lambda^\prime+\mu^\prime \notin u \Cdot C_\mathrm{fund}$ then $M(\lambda^\prime,\mu^\prime)$ is singular because the regular part
	\begin{align*}
	\big( L(\lambda^\prime) \otimes L(\mu^\prime) \big)_\mathrm{reg} & \cong \big( T_0^{u\Cdot\lambda^\prime} L(u\Cdot0) \otimes T_0^{\mu^\prime} L(0) \big)_\mathrm{reg} \\
	& \cong \bigoplus_{\nu \in C_\mathrm{fund} \cap X} T_0^\nu \big( L(u\Cdot0) \otimes L(0) \big)_\mathrm{reg}^{\oplus c_{u\Cdot\lambda^\prime,\mu^\prime}^\nu} \\
	& \cong \bigoplus_{\nu \in C_\mathrm{fund} \cap X} \big( T_0^\nu L(u\Cdot0) \big)^{\oplus c_{u\Cdot\lambda^\prime,\mu^\prime}^\nu} \\
	& \cong \bigoplus_{\nu \in C_\mathrm{fund} \cap X} L(u\Cdot\nu)^{\oplus c_{u\Cdot\lambda^\prime,\mu^\prime}^\nu}
	\end{align*}
	is a direct sum of simple $\G$-modules with highest weights in $u \Cdot C_\mathrm{fund}$; see Theorem \ref{thm:translationtensor}.
	As before, Theorem  \ref{thm:BowmanDotyMartin} implies that $M(\lambda^\prime,\mu^\prime)$ is a tilting module and therefore $M(\lambda^\prime,\mu^\prime) \cong T(\lambda^\prime+\mu^\prime)$.
	
	Finally, suppose that $\mu^\prime \in u\Cdot C_\mathrm{fund}$.
	By Lemma \ref{lem:regularpartA2} and Theorem \ref{thm:translationtensor}, we have
	\begin{align*}
	\big( L(\lambda^\prime) \otimes L(\mu^\prime) \big)_\mathrm{reg} & \cong \big( T_0^{u\Cdot\lambda^\prime} L(u\Cdot0) \otimes T_0^{u\Cdot\mu^\prime} L(u\Cdot0) \big)_\mathrm{reg} \\
	& \cong \bigoplus_{\nu \in C_\mathrm{fund} \cap X} T_0^\nu \big( L(u\Cdot0) \otimes L(u\Cdot0) \big)_\mathrm{reg}^{\oplus c_{u\Cdot\lambda^\prime,u\Cdot\mu^\prime}^\nu} \\
	& \cong \bigoplus_{\nu \in C_\mathrm{fund} \cap X} \big( T_0^\nu M(0) \oplus T_0^\nu L(0) \big)^{\oplus c_{u\Cdot\lambda^\prime,u\Cdot\mu^\prime}^\nu} \\
	& \cong \bigoplus_{\nu \in C_\mathrm{fund} \cap X} \big( M(\nu) \oplus L(\nu) \big)^{\oplus c_{u\Cdot\lambda^\prime,u\Cdot\mu^\prime}^\nu} ,
	\end{align*}
	and all singular indecomposable direct summands of $L(\lambda^\prime) \otimes L(\mu^\prime)$ are tilting modules by Theorem \ref{thm:BowmanDotyMartin}.
	On the other hand, we have $(\lambda^\prime+\mu^\prime+\rho,\alpha_\mathrm{h}^\vee)\geq 2\ell$ because $(\lambda^\prime+\rho,\alpha_\mathrm{h}^\vee)\geq\ell+1$ and $(\mu^\prime+\rho,\alpha_\mathrm{h}^\vee)\geq\ell+1$, hence
	\[ \lambda^\prime+\mu^\prime \notin C_\mathrm{fund} \cup u\Cdot C_\mathrm{fund} \cup us\Cdot C_\mathrm{fund} \cup ut\Cdot C_\mathrm{fund} \]
	and $L(\lambda^\prime+\mu^\prime)$ is not a composition factor of $M(\nu)$ or $L(\nu)$, for any $\nu \in C_\mathrm{fund} \cap X$.
	We conclude that $M(\lambda^\prime,\mu^\prime)$ is a tilting module, so $M(\lambda^\prime,\mu^\prime) \cong T(\lambda^\prime+\mu^\prime)$.
\end{proof}

\begin{Corollary} \label{cor:highestweightdirectsummandindecG1typeA2}
	Let $\lambda^\prime,\mu^\prime \in X_1$.
	Then $M(\lambda^\prime,\mu^\prime)$ is indecomposable as a $\G_1$-module.
\end{Corollary}
\begin{proof}
	By Lemma \ref{lem:highestweightdirectsummandtypeA2}, we have either
	\[ M(\lambda^\prime,\mu^\prime) \cong T(\lambda^\prime+\mu^\prime) \qquad \text{or} \qquad  M(\lambda^\prime,\mu^\prime) \cong L(\lambda^\prime+\mu^\prime) . \]
	In the second case, we further have $\lambda^\prime+\mu^\prime \in u \Cdot C_\mathrm{fund} \cap X \subseteq X_1$ and it follows that $L(\lambda^\prime+\mu^\prime)$ is simple as a $\G_1$-module.
	As $\lambda^\prime,\mu^\prime\in X_1$, we have
	\[ \lambda^\prime+\mu^\prime \in \big\{ \gamma\in X^+ \mathrel{\big|} (\gamma,\alpha^\vee) \leq 2\ell-2 \text{ for all } \alpha\in\Pi \big\} . \]
	This set of weights is disjoint from $ustus\Cdot \overline{C}_\mathrm{fund}$ and $utsut\Cdot \overline{C}_\mathrm{fund}$ (see Figure \ref{fig:alcovesA2}), so Remark~\ref{rem:indecomposableasG1moduletypeA2} implies that $T(\lambda^\prime+\mu^\prime)$ is indecomposable as a $\G_1$-module.
\end{proof}

Let us write $\lambda = \sum_{i\geq0} \ell^i \lambda_i$ and $\mu = \sum_{i\geq0} \ell^i \mu_i$, with $\lambda_i,\mu_i \in X_1$ for all $i\geq0$.

\begin{Corollary}
	We have $M(\lambda,\mu) \cong \bigotimes_{i\geq0} M(\lambda_i,\mu_i)^{[i]}$.
\end{Corollary}
\begin{proof}
	The $\G$-modules $M(\lambda_i,\mu_i)$, for $i\geq0$, are indecomposable as $\G_1$-modules by Corollary \ref{cor:highestweightdirectsummandindecG1typeA2} and the claim follows from Corollary \ref{cor:highestweightdirectsummandFrobenius}.
\end{proof}

By combining Theorem \ref{thm:exA2modular}, Lemma \ref{lem:highestweightdirectsummandtypeA2} and Corollary \ref{cor:highestweightdirectsummandFrobenius}, we obtain a complete description of the generic direct summands of tensor products of simple $\G$-modules, for $\G$ of type $\mathrm{A}_2$.

\subsection*{Dual Weyl modules}

Let us once again start by recalling the main results of X.~Chan and J.~Wang in more detail.
We return to our strategy of discussing the modular case and the quantum case at the same time.

According to Theorem 2.1(3) in \cite{ChenWang}, there is, for every $\nu \in C_\mathrm{fund} \cap X$, a unique indecomposable $\G$-module $M_\nabla(\nu)$ (denoted there by $Q(u\Cdot\nu)$) that admits a short exact sequence
\[ 0 \to \nabla(u\Cdot\nu) \to M_\nabla(\nu) \to \nabla(us\Cdot\nu) \oplus \nabla(ut\Cdot\nu) \to 0 . \]
See also Lemma 3.1 in \cite{ChenWang} and the computation thereafter.
Furthermore, the $\G$-module $M_\nabla(\nu)$ is isomorphic to the injective hull $I_\pi(u\Cdot\nu)$ of the simple $\G$-module $L(u\Cdot\nu)$ in the Serre subcategory of $\Rep(\G)$ generated by the simple $\G$-modules with highest weights in $\pi = \{ \nu , u\Cdot\nu , us\Cdot\nu , ut\Cdot\nu \}$.
Now (a shortened version of) Theorem~2.1 in \cite{ChenWang} is as follows:

\begin{Theorem} \label{thm:ChenWang}
	Let $\lambda,\mu\in C_\mathrm{fund} \cap X$.
	\begin{enumerate}
	\item The tensor product $\nabla(\lambda) \otimes \nabla(\mu)$ is a direct sum of indecomposable tilting modules.
	\item The tensor product $\nabla(u\Cdot\lambda) \otimes \nabla(\mu)$ is a direct sum of induced modules $\nabla(u\Cdot\nu)$, with $\nu \in C_\mathrm{fund} \cap X$, and of negligible tilting modules.
	\item The tensor product $\nabla(u\Cdot\lambda) \otimes \nabla(u\Cdot\mu)$ is a direct sum of $\G$-modules of the form $M_\nabla(\nu)$, with $\nu \in C_\mathrm{fund} \cap X$, and of negligible tilting modules.
	\end{enumerate}
\end{Theorem}

Using the preceding theorem, it is straightforward to work out the generic direct summands of tensor products of induced modules with $\ell$-restricted $\ell$-regular highest weights. First note that
\[ \nabla(0) \cong \nabla(0) \otimes \nabla(0) \cong \big( \nabla(0) \otimes \nabla(0) \big)_\mathrm{reg} \cong G_\nabla(e,e) \]
and similarly
\[ \nabla(u\Cdot0) \cong \nabla(u\Cdot0) \otimes \nabla(0) \cong \big( \nabla(u\Cdot0) \otimes \nabla(0) \big)_\mathrm{reg} \cong G_\nabla(u,e) , \]
even without using the theorem. The generic direct summand $G_\nabla(u,u)$ of $\nabla(u\Cdot0) \otimes \nabla(u\Cdot0)$ is regular and belongs to $\Rep_0(\G)$, so part (3) of Theorem \ref{thm:ChenWang} implies that
\[ G_\nabla(u,u) \cong M_\nabla(0) . \]

\begin{Remark}
	In Corollary 3.15 in \cite{ChenWang}, it is shown that, for $\lambda,\mu,\nu \in C_\mathrm{fund} \cap X$, the multiplicity of $M_\nabla(\nu)$ in a Krull-Schmidt decomposition of $\nabla(u\Cdot\lambda) \otimes \nabla(u\Cdot\mu)$ is given by the structure constant $c_{\lambda,\mu}^\nu$ of the Verlinde algebra.
	The idea that these structure constants should govern the multiplicities of regular indecomposable direct summands in tensor products (as in Theorem \ref{thm:translationtensor}) arose when the author was studying this result, but the proof from \cite{ChenWang} does not carry over to our more general setting.
	A similar argument to X.~Chan and J.~Wang's proof can, however, be used to show that $c_{\lambda,\mu}^\nu$ is the multiplicity of $T_0^\nu G_\nabla(x,y)$ in a Krull-Schmidt decomposition of $\nabla(x\Cdot\lambda) \otimes \nabla(y\Cdot\mu)$, for $x,y\in W_\mathrm{aff}^+$.
	(In our setting, this follows from Remarks \ref{rem:genericdirectsummandinducedmodule} and \ref{rem:genericdirectsummand}.)
\end{Remark}

\bibliographystyle{alpha}
\bibliography{tensor}

\end{document}